\documentclass[11pt]{article}
\usepackage{amsmath,amssymb,amsthm}
\usepackage{graphicx,color}
\graphicspath{{figures/}}
\usepackage[ %
a4paper, %
nohead, %
ignorefoot, %
twoside, %
left=2.5cm, %
right=2.cm, %
top=2cm, %
bottom=2cm %
]{geometry} %
\usepackage[font=small,labelfont=bf]{caption}
\newcommand{\BMHC}{}
\newcommand{\EMHC}{}
\newcommand{\BC}{}
\newcommand{\EC}{}

\newcommand{\texorpdfstring}[2]{#1}   
\usepackage{hyperref}%
\definecolor{gray}{rgb}{0.2,0.2,0.2}
\hypersetup{%
  unicode=true,          
  colorlinks=true,       
  linkcolor=gray,          
  citecolor=gray      
}
\theoremstyle{plain}
\newtheorem{Thm}{Theorem}[]

\newtheorem{Lem}[Thm]{Lemma}

\newtheorem{Ass}[Thm]{Assumption}

\DeclareMathOperator*{\sinc}{sinc}
\DeclareMathOperator*{\sech}{sech}
\DeclareMathOperator*{\id}{id}
\def\longrightharpoonup{
\relbar\joinrel\joinrel\relbar\joinrel\joinrel\relbar\joinrel\joinrel\relbar\joinrel\joinrel\relbar\joinrel\joinrel\relbar\joinrel\joinrel\rightharpoonup}
\newcommand{\xrightharpoonup}[1]{\stackrel{#1}{\longrightharpoonup}}
\begin{document}
\title{KdV-like solitary waves in two-dimensional FPU-lattices }
\author{Fanzhi Chen\thanks{{\tt{fanzhi.chen@uni-muenster.de}}, University of M\"unster, Germany } \and Michael Herrmann\thanks{{\tt{michael.herrmann@tu-braunschweig.de}}, Technische Universit\"at Braunschweig, Germany }}
\date{\today}
\maketitle
\begin{abstract} 
We prove the existence of solitary waves in the KdV limit of two-dimensional FPU-type lattices using asymptotic analysis of nonlinear and singularly perturbed integral equations. In particular, we generalize the existing results by Friesecke and Matthies since we allow for arbitrary propagation directions and non-unidirectional wave profiles.
\end{abstract}
\quad\newline%
 \begin{minipage}[t]{0.15\textwidth}%
Keywords: 
 \end{minipage}%
 \begin{minipage}[t]{0.8\textwidth}%
\emph{two-dimensional FPU-lattices}, \emph{KdV limit of lattice waves},\\
\emph{asymptotic analysis of singularly perturbed integral equations} 
 \end{minipage}%
 \medskip
 \newline\noindent
 \begin{minipage}[t]{0.15\textwidth}%
   MSC (2010): %
 \end{minipage}%
 \begin{minipage}[t]{0.8\textwidth}%
37K60,  
37K40,  
74H10  	
\end{minipage}%
{\scriptsize \tableofcontents}
\section{Introduction}
\label{sect:intro}
Since the pioneering work \cite{FPU55} \BMHC a lot of mathematical research has \EMHC  been devoted to one-dimensional FPU lattices, of which the only nontrivial completely integrable type is the so-called Toda lattice, that admits explicit traveling wave solutions in terms of Jacobian elliptic functions, see for instance \cite{FM15}.  Among these works \cite{ZK65} was the first to reveal a connection between solitary waves with long-wave-lengths and small amplitudes and the Korteweg-de Vries (KdV) equation. By diminishing the lattice spacing the authors were able to pass to a continuum described by a KdV equation, which is a completely integrable PDE. However, since their analysis was purely formal, it didn't lead to a rigorous proof. The existence of KdV-like traveling waves in FPU chains was justified in a rigorous way much later, in the first of a series of papers \cite{FP99}, \cite{FP02}, \cite{FP04a} and \cite{FP04b}, of which the other three deal with the stability theory. The proof was based on the introduction of a new framework for passing to the continuum and concerned only lattices with \BMHC interactions \EMHC between nearest \BMHC neighbors, see also \cite{FM15} for related results concerning periodic waves. The \EMHC existence result from \cite{FP99} has been extended in \EMHC \cite{HML15} to 1D lattices with nonlocal interactions, \BC replacing \EC the original analysis \BMHC of Fourier poles \EMHC in the complex plane by direct estimates on the real line. \BMHC The KdV approximation can also be used to derive modulation equations for the large-scale dynamics of well-chosen initial data. We refer to \cite{SW00,HW08,HW09,HW13} for results in the classical FPU setting, and to \cite{GMWZ14} for more complex atomic systems.\EMHC

Existence results on various types of waves in FPU chains can be found in the literature. \BMHC The paper \EMHC \cite{Ioo00} provides a precise description of all subsonic and supersonic waves with small amplitudes via the center manifold reduction. In \cite{FW94} the existence of solitary waves with large amplitudes was proved for the first time by applying the concentration-compactness principle, see also \cite{Che13} for similar results under slightly different conditions and \BMHC \cite{FV99,Her10} for another variational approach yielding unimodal waves. \EMHC In all these large-amplitude results the wave speed is in principle unknown. However, using  \BMHC mountain pass techniques as in \cite{SW97,Pan05} one also finds periodic and solitary waves with prescribed wave speed above a critical threshold. Finally,
for certain potentials there even exist several types of heteroclinic waves, see \BMHC \cite{HR10,TV05,SZ12}.\EMHC

The investigations mentioned above concern exclusively 1D FPU lattices. In the practice of physics, two or more dimensional FPU-lattices are more relevant, since they provide simplified models for crystals \BMHC and solids. \EMHC However, to the best of our knowledge, no existence result for KdV-like waves in such FPU lattices has been obtained except for a degenerate case \cite{FM03}, where it is assumed that the wave is unidirectional and \BMHC longitudinal as it \EMHC moves along the horizontal direction. This assumption reduces the original two-dimensional system to a one-dimensional one and the geometric nonlinearity, \BMHC which arises from the linear stress-strain relation \EMHC due to the introduction of diagonal springs, \BMHC enables \EMHC one to apply the results from \cite{FP99}.

In the present paper we show an existence result for 2D FPU lattices in a general framework, which covers different lattice geometries and allows for arbitrary \BMHC propagation \EMHC directions. Our asymptotic approach adapts some arguments developed in \cite{HML15} but is more sophisticated due to the two-dimensional setting. The discussion of \BMHC the 2D case \EMHC in the present paper hints at similar results for 3D \BMHC lattices, \EMHC which have to be left for further investigations. Besides, the study of stability of the resulting solutions requires ideas lying beyond the scope of the present paper and \BMHC will be \EMHC completely omitted. 
\subsection{Setting of the problem}\label{sect:Overview}
\BMHC In this subsection we study the square lattice --- see the left panel of Figure \ref{Fig:AllLattices} --- as  prototypical example and show that any lattice wave complies with the nonlocal advance-delay-differential equation formulated in \eqref{ddequ}. Other lattice geometries produce basically the same equation but the number of coupling terms $M$, 
the coupling parameters $k_m$, and the effective nonlinearities $F^m_i$ will depend on the mechanical details in the lattice, see \S\ref{sect:app}.\EMHC 
\par
\begin{figure}[ht]
\centering{%
\includegraphics[width=0.95\textwidth]{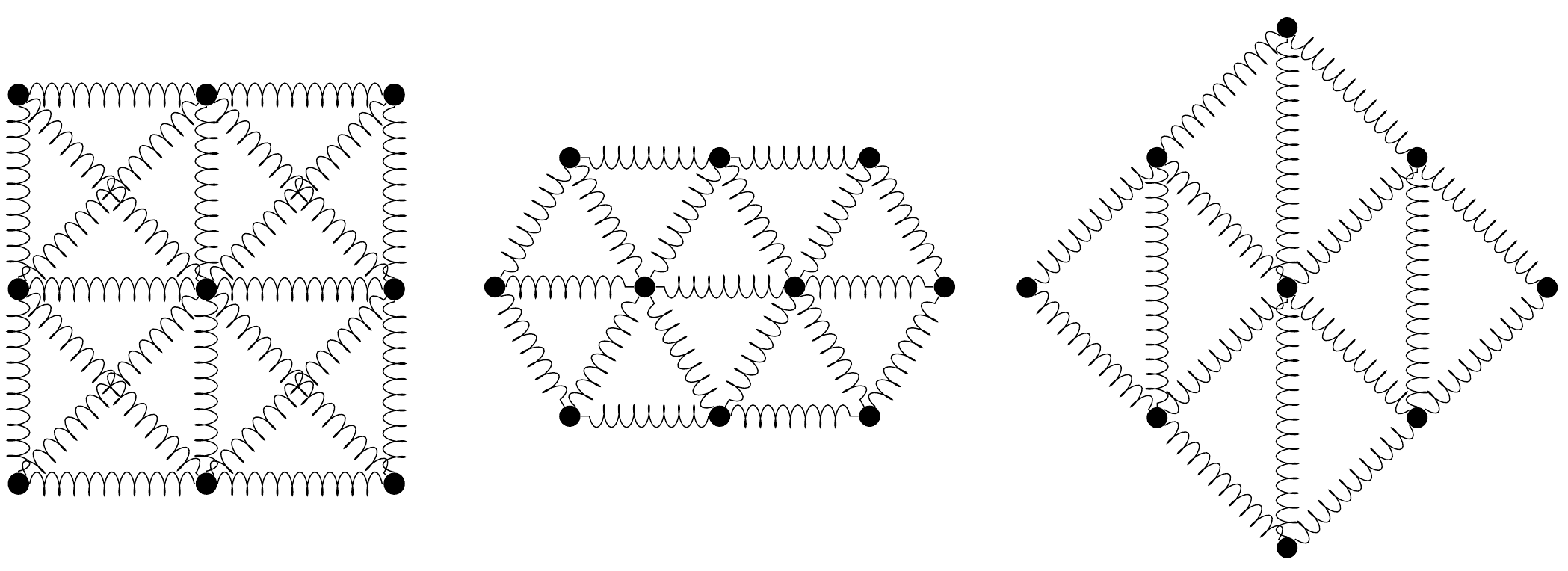}
}%
\caption{\textit{Left panel}: \BMHC Cartoon of the square lattice. The vertical and the horizontal springs are described by the potential $V_1$, while all diagonal springs correspond to $V_2$. \textit{Center panel}: Triangle lattice with identical springs and single potential function $V$. \textit{Right panel}: \BMHC Cartoon of the diamond lattice, which can be regarded as a square lattice without horizontal springs. The lattices have different symmetry groups and produce different coupling terms in the advance-delay-differential equation for lattice waves, see \eqref{ddequ}. \EMHC }
\label{Fig:AllLattices}
\end{figure}
We assume that the lattice is indexed by $(i,j)\in \mathbb{Z}^2$ and the position of the $(i,j)^{\text{th}}$ particle is given by 
\begin{align*}
\begin{pmatrix}
r_*i\\
r_*j
\end{pmatrix}+q_{ij}(t),
\end{align*}
where $r_*>0$ is a reference lattice parameter and $q_{ij}(t)=(q_{ij,1}(t),q_{ij,2}(t))^T\in \mathbb{R}^2$ represents the displacement of the  $(i,j)^{\text{th}}$ particle at time $t$. By summing up the forces exerted on a given particle by its eight neighbors and applying the second Newton's law we obtain
\begin{align}\label{motion}
\begin{split}
\ddot{q}_{ij}=&+\BMHC(\nabla\phi_1)(q_{i+1j}-q_{ij})-(\nabla\phi_1)(q_{ij}-q_{i-1j})
\\&+(\nabla\phi_2)(q_{ij+1}-q_{ij})-(\nabla\phi_2)(q_{ij}-q_{ij-1})\\
&+(\nabla\phi_3)(q_{i+1j+1}-q_{ij})-(\nabla\phi_3)(q_{ij}-q_{i-1j-1})
\\&+(\nabla\phi_4)(q_{i+1j-1}-q_{ij})-(\nabla\phi_4)(q_{ij}-q_{i-1j+1}),\EMHC
\end{split}
\end{align} 
where
\begin{align*}
\phi_1(x_1,x_2)& =V_1\big(\sqrt{(x_1+r_*)^2+x_2^2}\,\big), \ \ \ \ \phi_2(x_1,x_2)=V_1\big(\sqrt{x_1^2+(x_2+r_*)^2}\,\big), \\
\phi_3(x_1,x_2)& =V_2\big(\sqrt{(x_1+r_*)^2+(x_2+r_*)^2}\,\big),\ \phi_4(x_1,x_2)=V_2\big(\sqrt{(x_1+r_*)^2+(x_2-r_*)^2}\,\big)
\end{align*}
are the four effective potentials corresponding to the horizontal, vertical, and diagonal springs, respectively. In this paper we are \BMHC solely \EMHC interested in traveling wave solutions of the following type
\begin{align}\label{travel}
q_{ij}(t)=\epsilon \BMHC Q_{\epsilon}\EMHC\Big(\epsilon\big(\kappa_1i+\kappa_2j-c_{\epsilon}t\big)\Big) \in \mathbb{R}^2,
\end{align}
 where $\kappa:=(\kappa_1,\kappa_2)^T$ is the wave vector prescribing the propagating direction of the wave and $c_{\epsilon}>0$ the wave speed that depends on the \BMHC scaling \EMHC parameter $\epsilon>0$. By a suitable scaling we can assume that $\kappa$ is normalized with respect to its length, \BMHC and denoting the angle between $\kappa$ and the positive axis by $\alpha$, we write \EMHC $\kappa=(\cos(\alpha),\sin(\alpha))^T$.
\par
Substituting \BMHC the \EMHC ansatz \eqref{travel} into \eqref{motion}, we arrive at a system of two difference-differential equations for the \BMHC unknown function $Q_\epsilon$, which can be written in abstract form as \EMHC 
\begin{align}\label{ddequ}
\epsilon^3 c_{\epsilon}^2 \BMHC Q_{\epsilon}^{\prime\prime}\EMHC=\sum_{m=1}^Mk_m\epsilon\delta_{-k_m\epsilon} F^m\big(k_m\epsilon^2\delta_{+k_m\epsilon}\BMHC Q_{\epsilon}\EMHC\big).
\end{align}
\BMHC Here $M=4$ denotes the number of interacting neighbors (divided by two),  the coupling constants depend on $\alpha$ via
\begin{align*}
k_1=\kappa_1,\qquad k_2=\kappa_2,\qquad k_3=\kappa_1+\kappa_2,\qquad  k_4=\kappa_1-\kappa_2,
\end{align*}
and 
\begin{align*}
F^m(x_1,x_2)=\nabla \phi_m(x_1,x_2)-\nabla\phi_m(0,0)
\end{align*}
abbreviates the (normalized) gradient of the effective potential related to $k_m$. Moreover,
\EMHC the right hand side of \eqref{ddequ} is formulated in terms of the difference operators
 \begin{align*}
\delta_{+\eta}Y:=\frac{Y(\cdot+\eta)-Y(\cdot)}{\eta},\qquad \delta_{-\eta}Y:=\frac{Y(\cdot)-Y(\cdot-\eta)}{\eta},\,\qquad \eta>0,
\end{align*}
\BMHC which reflect that the problem involves for each \BC $m=1,...,M$ \EC precisely one advance and one delay term and allow us  to reformulate \eqref{ddequ} in \S\ref{sect:abstract} as an integral equation for the velocity profile \EMHC 
\begin{align}
\label{Eqn:DefWEps}
W_{\epsilon}=Q^\prime_{\epsilon}\,.
\end{align}
\par
\BMHC Recall that the special case $\alpha=0$ has been studied in \cite{FM03}. In particular, it has been shown that any wave propagating along the horizontal lattice direction must be unidirectional and longitudinal. This means that any reasonable solution to \eqref{ddequ} must satisfy $W_{\epsilon,2}=0$ and the original 2D problem can hence be reduced to a one-dimensional one. Moreover, the identities $k_2=0$ and $k_1=k_3=k_4$ imply that the different coupling terms can be subsumed into a single one. For general $\alpha$, however, both simplifications are no longer possible and the \BC lattice waves \EC are asymptotically unidirectional in the sense that $W_{\epsilon,2}(\xi)/W_{\epsilon,1}(\xi)=\mathrm{const}$ manifests only in the limit $\epsilon\to0$, see Figure~\ref{Fig:NumSim}.
\par
\BMHC It has also been observed in \cite{FM03} \EMHC that the diagonal springs lead to the nonlinearity that is necessary for the existence of solitary waves. \BMHC This holds even in the case of linear spring forces, i.e., if the physical interactions potentials $V_i$ are quadratic (geometric nonlinearity). \EMHC Without diagonal springs there would be no resistance against the sheer forces and the mechanical structure would become unstable. 
\begin{figure}[ht!]
\centering{
   \includegraphics[width=0.4\textwidth]{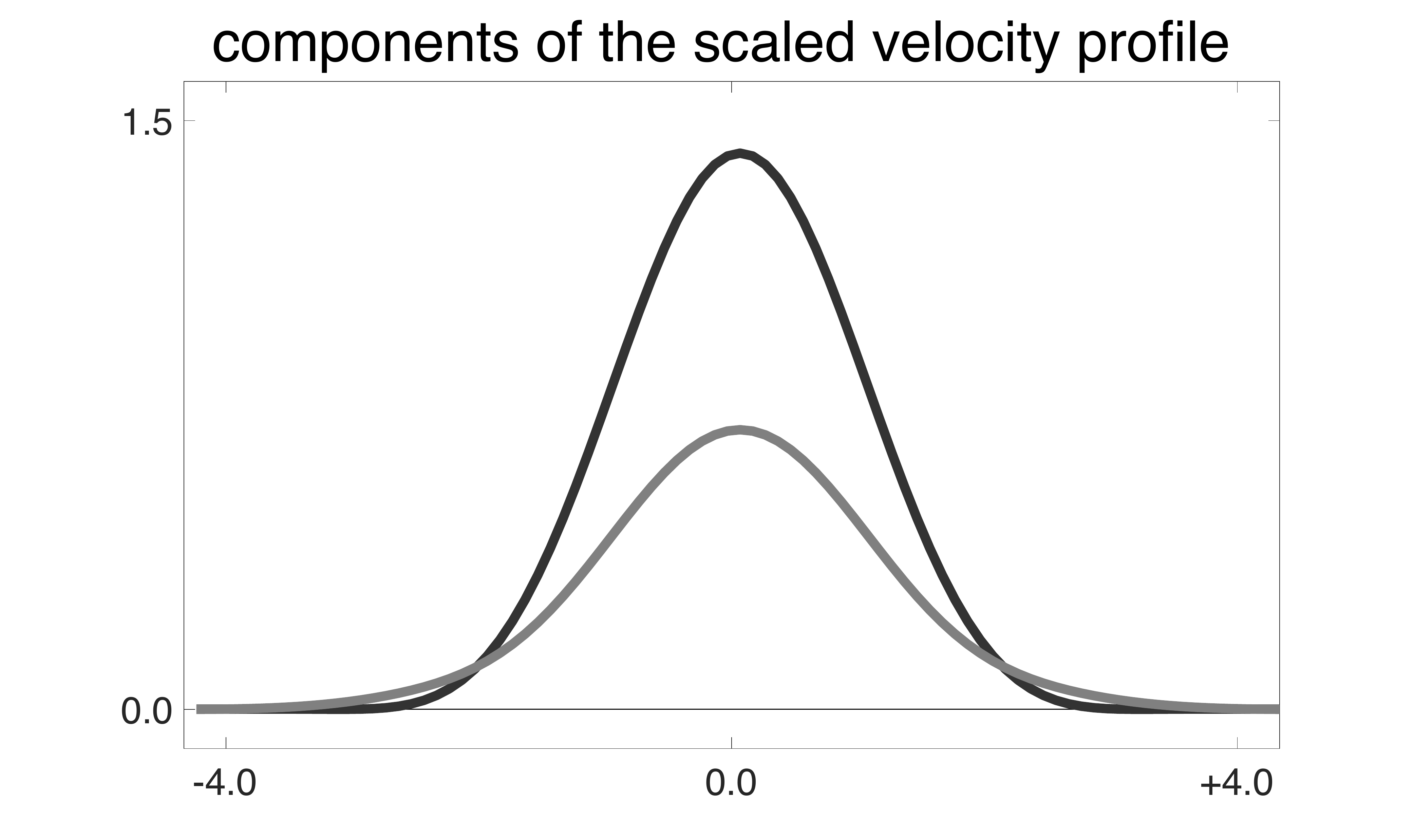}
   \includegraphics[width=0.4\textwidth]{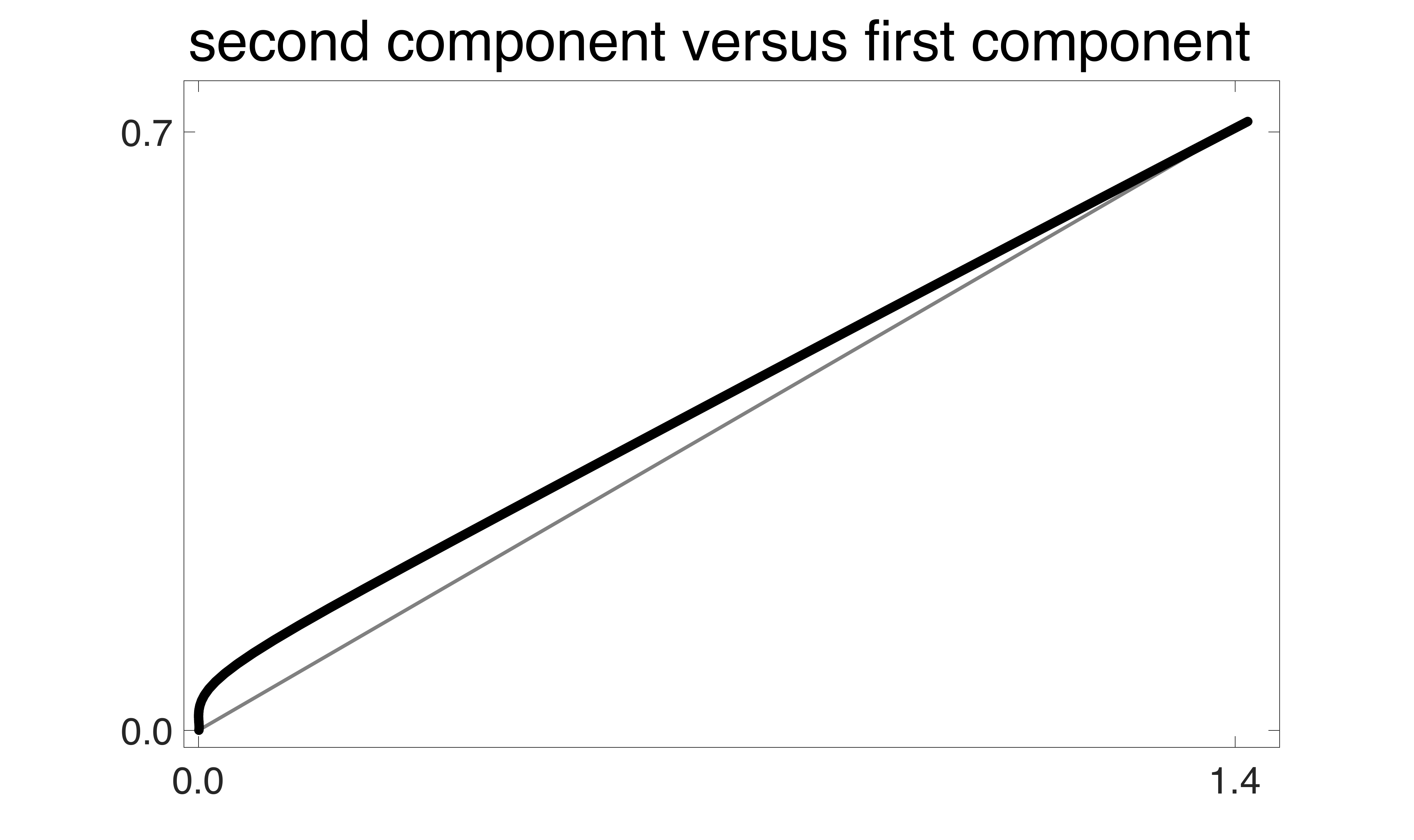}
} %
\caption{\textit{Left panel:} Numerical approximations of $W_{\epsilon,1}(\xi)$ (black) and $W_{\epsilon,2}(\xi)$ (gray) for the square lattice with angle $\alpha=\frac{\pi}{8}$ and
positive $\epsilon$. \textit{Right panel:} The plot $W_{\epsilon,2} $ versus $W_{\epsilon,1}$ reveals that the two components of $W_{\epsilon}$ are not proportional, which means that $W_{\epsilon}$ is not unidirectional and our problem cannot be reduced to a one-dimensional one as in \cite{FM03}.}\label{Fig:NumSim}
\end{figure}
\subsection{Proof strategy and main result}

\BMHC In this paper we follow the fundamental ideas behind KdV limits and construct solutions to \eqref{ddequ} in the asymptotic regime of small $\epsilon>0$, i.e., we suppose that the lattice waves have small amplitudes and large wave-lengths. This is illustrated in Figure \ref{Fig:ScaledWaves} and enables us to employ Taylor expansion and to regard the difference terms as approximate differential operators. To this end we assume that the nonlinearities \EMHC $F_i^m$, which take over the role of the partial derivatives of effective potentials in the lattices, have the form
\begin{equation}\label{Ide1}
F_i^m(x_1,x_2)=\alpha_{i,1}^mx_1+\alpha_{i,2}^mx_2+\tfrac{1}{2}\big(\beta_{i,11}^mx_1^2+2\beta_{i,12}^mx_1x_2+\beta_{i,22}^mx_2^2\big)+\Psi_i^m(x_1,x_2)
\end{equation}
for all $i=1,2$ and $m=1,...,M$, where $\Psi_i^m(x_1,x_2)$ represent the higher order terms
\BMHC and are hence asymptotically small perturbations of the linear and the quadratic contributions. \EMHC
\par

\begin{figure}[ht!]
\centering{ %
   \includegraphics[width=0.85\textwidth]{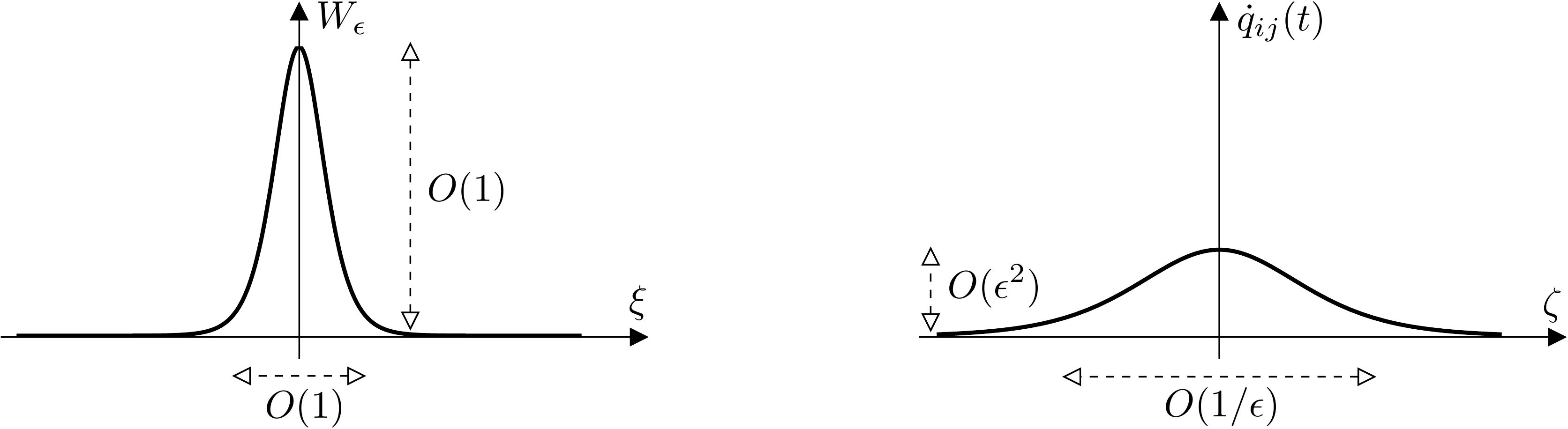} %
} %
\caption{\textit{Left panel}. \BMHC Scaled velocity profile $W_{\epsilon}$ as function of $\xi$. \textit{Right panel}. Cartoon of the atomistic velocities in the corresponding KdV wave, where $\zeta=\kappa_1i+\kappa_2j-c_{\epsilon}t$ denotes the phase with respect to the original variables. The unscaled profile is obtained from the scaled one by stretching the argument by ${1}/{\epsilon}$ and pressing the amplitude by $\epsilon^2$. \EMHC
}\label{Fig:ScaledWaves}
\end{figure}
Following \BMHC \cite{HML15} we \EMHC consider the \BMHC vector-valued velocity profile $W_{\epsilon}$ from \eqref{Eqn:DefWEps} \EMHC and decompose it as the sum of a limit profile $W_0$ and an $O(\epsilon^2)$-corrector \BMHC via\EMHC
\begin{equation}\label{decompos}
W_{\epsilon}:=W_0+\epsilon^2V_{\epsilon}\,,
\end{equation}
\BC 
where the small parameter $\epsilon>0$ quantifies via
\begin{equation}\label{speed}
\sigma_{\epsilon}:=c_{\epsilon}^2=\sigma_0+\epsilon^2
\end{equation}
how much the squared wave speed exceeds a given reference value $\sigma_0$, which \EC \BMHC is completely determined by the linear coefficients $\alpha_{i,j}^m$ \BMHC from \eqref{Ide1}, \EMHC see formula \eqref{defsig}. Applying \eqref{speed} and \BMHC gathering \EMHC linear terms with respect to $W_{\epsilon}$ in \eqref{equep1} on the left hand side and the \BMHC nonlinear \EMHC terms on the right, we arrive at the equation
  \begin{align}\label{BW}
\BMHC \mathcal{B}_{\epsilon}W_{\epsilon}\EMHC =\mathcal{Q}_{\epsilon}[W_{\epsilon}]+\epsilon^2\mathcal{P}_{\epsilon}[W_{\epsilon}].
\end{align}
Here, $\mathcal{B}_{\epsilon}$ is a linear integral operator introduced in \eqref{deflam}, $\mathcal{Q}_{\epsilon}$ depends quadratically on $W_{\epsilon}$, and $\mathcal{P}_{\epsilon}$ stems from the higher order terms $\Psi_i^m$, \BMHC see \eqref{Eqn:DefQeps} and \eqref{Eqn:DefPeps}. \EMHC The key observation \BMHC for small $\epsilon>0$ \EMHC is that the formal limit equation 
\begin{align}\label{BWLim}
\BMHC\mathcal{B}_{0}W_{0}\EMHC =\mathcal{Q}_{0}[W_{0}]\,,
\end{align}
\BMHC which will be identified in \eqref{Eqn:LeadingOrder1} and \eqref{Eqn:LeadingOrder2}, \EMHC is equivalent to
 \begin{align}
 \label{Ws1}
W_{0,1}=W_*,\qquad W_{0,2}=\lambda W_*,
\end{align}
\BMHC where the non-trivial and decaying function $W_*$ is the unique even and homoclinic solution to the ODE\EMHC
 \begin{align}\label{Ws}
W''_*=d_1W_*-d_2W_*^2.
\end{align}
\BMHC The constants $d_1$ and $d_2$  are positive and can be computed explicitly, see \BMHC \eqref{deflam} and \eqref{defd}. They depend on \EMHC the linear and quadratic coefficients $\alpha_{i,j}^m$ and $\beta_{i,jk}^m$, \BMHC as well as on the scalar quotient \EMHC $\lambda=W_{0,2}/W_{0,1}$. Moreover, \EMHC  $W_*$ defines via $w(t,\xi):=W_*(\xi-t)$ the solitary wave for the KdV equation
\begin{align*}
d_1\partial_tw+d_2\partial_{\xi}w^2+\partial_{\xi}^3w=0.
\end{align*} 
At the next level, we \BMHC turn \EMHC \eqref{Ws} into a fixed-point equation with respect to the corrector $V_{\epsilon}$ by inserting \eqref{decompos} and rearranging the terms in such a way, that the linear terms with respect to $V_{\epsilon}$ stand on the left hand side, which yield automatically a linear operator $\mathcal{L}_{\epsilon}$. After the uniform inversion of $\mathcal{L}_{\epsilon}$ on $(\mathsf{L}_{\rm{even}}^2(\mathbb{R}))^2$, see Theorem \ref{inverL}, \BMHC we finally deduce the fixed point equation
\EMHC 
 \begin{align}
\label{Eqn:FixedPoint}
V_{\epsilon}=\mathcal{F}_{\epsilon}[V_{\epsilon}]\,,
\end{align}
where \BMHC the operator $\mathcal{F}_\epsilon$ \EMHC will be identified in \eqref{endeq}.
The contraction property of $\mathcal{F}_{\epsilon}$ will \BMHC  be shown \EMHC in a sufficiently large ball in $(\mathsf{L}_{\rm{even}}^2(\mathbb{R}))^2$, \BMHC and the \EMHC Banach fixed-point theorem provides us with a solution, which is at the same time unique in the chosen ball. \BMHC Notice that \BC uniqueness of large-amplitude traveling waves \EC \BMHC is a notoriously difficult and open problem in the mathematical theory of Hamiltonian lattice waves.  \EMHC
\par
\BMHC The sketched asymptotic approach shares many similarities with the discussion in \cite{HML15}, but all arguments have to be adapted to the two-dimensional situation. For instance, our version of $\mathcal{B}_{\epsilon}$ is not symmetric and its inversion involves the study of a matrix of operators instead of a single operator on $\mathsf{L}^2(\mathbb{R})$. Moreover, the fundamental operators and scalar quantities depend on the propagation angle $\alpha$. \EMHC
\par
\BMHC Our main result guarantees that the abstract advance-delay-differential equation \eqref{ddequ} admits  KdV-like solutions for sufficiently small $\epsilon>0$ provided that 
the nonlinearities $F^m_i$ and the parameters $k_m$ satisfies certain requirements. These conditions, which can be checked as described in \S\ref{sect:app}, encode the particular lattice geometry, the propagation direction of the wave, and the choice of the physical potentials. \EMHC 
\begin{Thm}\label{maintm}
Under Assumptions \BMHC \ref{Ass sym}, \ref{Ass2},  \ref{Ass1}, and \ref{Ass inverse} there exists a sound speed $\sigma_0$ along with a KdV traveling wave $W_0$ as in \eqref{Ws1} and \eqref{Ws} \EMHC such that for any sufficiently small $\epsilon>0$ and wave speed $c_{\epsilon}=\sqrt{\sigma_0+\epsilon^2}$, equation (\ref{ddequ}) has a solution whose  velocity profile $W_{\epsilon}=Q_{\epsilon}^\prime\in (\mathsf{L}_{\rm{even}}^2(\mathbb{R}))^2$ satisfies $||W_{\epsilon}-W_0||_2\leq C\epsilon^2$.
\end{Thm}

This paper is organized as follows: \BMHC Within the sections \S\ref{sect:reform1}--\S\ref{sect:reform4} we reformulate  \eqref{ddequ} first as a nonlinear eigenvalue equation for $W_\epsilon$, and afterwards as fixed problem for the corrector $V_{\epsilon}$, see \eqref{Eqn:FixedPoint}. We also study the leading order equation 
\eqref{BWLim} and specify our assumptions concerning the Taylor coefficients from \eqref{Ide1} and their dependence on the propagation direction. \BMHC Using the auxiliary results from \S\ref{sect:invB} and \S\ref{sect:Aux2} we prove in \S\ref{sect:invL} the key asymptotic result, i.e., we show that the operator $\mathcal{L}_{\epsilon}$, which stems from the linearization of (\ref{BW}) around  $W_0$,  is uniformly invertible  in the $\mathsf{L}^2$-space of even functions. The contraction property of the fixed point operator $\mathcal{F}_{\epsilon}$ is then verified in \S\ref{sect:FP} and concludes the proof of our main result.  Finally, in \S\ref{sect:app} we discuss the three different lattices from Figure \ref{Fig:AllLattices}  as applications of Theorem \ref{maintm}. \EMHC 
\section{Abstract existence result}\label{sect:abstract}

\BMHC In this section we write (\ref{ddequ}) as a fixed point equation with respect to the corrector $V_{\epsilon}$ in \eqref{decompos} and prove under certain assumptions on the nonlinearities the existence of KdV-type solutions. To this end we denote by $C$ any generic constant that is independent of $\epsilon$ but can depend on $M$, the constants $k_m$, and the properties of the functions $F_{i}^m$. Moreover, we write $||\cdot||_p$ and $||\cdot||_{l,p}$ for the usual norms in the Lebesgue and Sobolev spaces $\mathsf{L}^p(\mathbb{R})$ and $\mathsf{W}^{l,p}(\mathbb{R})$, respectively, and $\langle \cdot,\cdot\rangle$ means the usual dual paring of functions. The same symbols also abbreviate the corresponding norms of vector-valued functions.
\EMHC
\subsection{Reformulation as integral equation}\label{sect:reform1}
\BMHC For \EMHC an arbitrary constant $\eta>0$, we define the integral operator $\mathcal{A}_{\eta}$  by 
\begin{equation}
\label{Def:OperatorAeta}
(\mathcal{A}_{\eta} Y)(\xi):=\frac{1}{\eta}\int_{\xi-\frac{\eta}{2}}^{\xi+\frac{\eta}{2}}Y(\tilde{\xi})d\tilde{\xi}\qquad\text{with}\quad Y\in \mathsf{L}^2(\mathbb{R}).
\end{equation}
 \BMHC This operator corresponds in Fourier space to the multiplication with the symbol function 
\begin{align}
\label{Eqn:SymbolA}
z\;\mapsto\;\sinc(\eta z/2)=\frac{\sin(\BC\eta z/2\EC)}{\eta z\BC/2\EC},
\end{align} and can \EMHC be formally expanded as follows 
\begin{equation}\label{formal}
\mathcal{A}_{\eta}=\id+\frac{\eta^2}{24}\partial^2+O(\eta^4).
\end{equation}
\BMHC In particular, \EMHC $\mathcal{A}_{\eta}$ is a singular perturbation of the identity due to the remainder terms with higher derivatives and this complicates the asymptotic analysis. Fortunately, the derivation of the decisive properties of the auxiliary operator $\mathcal{B}_{\epsilon}$, see (\ref{deflam}), \BMHC only \EMHC involves the inversion of the \BMHC Fourier \EMHC symbol function 
\begin{equation*}
1+\sum_{m=1}^M \gamma_m\frac{1-\sinc(k_m z/2)^2}{\epsilon^2},
\end{equation*}
 which corresponds to 
\begin{equation*}
\Big(\id+\sum_{m=1}^M \gamma_m\frac{\id-\mathcal{A}_{k_m}^2}{\epsilon^2}\Big)^{-1}.
\end{equation*}
The asymptotic properties of such an operator are well understood, \BMHC see for instance \cite{HML15} and the discussion in \S\ref{sect:invB}. \EMHC
\par
As already mentioned in \S\ref{sect:intro}, the operator $\mathcal{A}_{\epsilon}$ allows our problem to be reformulated as one concerning the velocity profile $W_{\epsilon}$ and this leads to a new framework for solving (\ref{ddequ}).
\begin{Lem}[\BMHC reformulation of traveling wave equation\EMHC]
\BMHC  Suppose $W_{\epsilon}$ from \eqref{Eqn:DefWEps} belongs to \EMHC $(\mathsf{L}^2(\mathbb{R}))^2$. Then the traveling wave equation (\ref{ddequ}) is equivalent to the nonlinear eigenvalue problem
 \begin{align}\label{equep1}
\epsilon^2 \sigma_{\epsilon}W_{\epsilon}&=\sum_{m=1}^{M}k_m\mathcal{A}_{k_m\epsilon}F^m(\epsilon^2 k_m\mathcal{A}_{k_m\epsilon}W_{\epsilon,1},\epsilon^2 k_m\mathcal{A}_{k_m\epsilon}W_{\epsilon,2})
\end{align} 
with \BMHC scaled eigenvalue \EMHC \BC $\sigma_\epsilon=c_\epsilon^2$  as in \eqref{speed} \EC \BMHC and unknown eigenfunction $W_{\epsilon}$.\EMHC
\end{Lem}
\begin{proof}
Suppose \eqref{equep1} \BMHC is satisfied and let \BMHC $Q_{\epsilon}(\xi):=\int_{0}^{\xi}W_{\epsilon}(\zeta)d\zeta$. \BMHC Then we can replace any term $(\mathcal{A}_{\eta}W_{\epsilon})(\xi)$ in the argument of the nonlinearities by $(\delta_{+\eta}Q_{\epsilon})(\xi-\eta/2)$ and differentiating both sides of \eqref{equep1} we obtain \eqref{ddequ} due to $(\mathcal{A}_\eta Y)^\prime(\xi)=(\delta_{-\eta}Y)(\xi+\eta/2)$. Conversely, integrating \eqref{ddequ} we obtain immediately \eqref{equep1} plus a constant vector $\in \mathbb{R}^2$, but this constant of integration vanishes due to $W_{\epsilon}\in (\mathsf{L}^2(\mathbb{R}))^2$ and $F^m((0,0))=(0,0)$ for all $m$.\EMHC
\end{proof}
\BMHC In our subsequent analysis we use the following properties of $\mathcal{A}_{\eta}$, where (\ref{rigorexp}) justifies  \BMHC the formal expansion (\ref{formal}) of $\mathcal{A}_{\eta}$ in a rigorous way.\EMHC
\begin{Lem}[\BMHC properties of the integral operator $\mathcal{A}_\eta$\EMHC ]\label{Prop A}
For each $\eta>0$, the integral operator $\mathcal{A}_{\eta}$ has the following properties:
\begin{enumerate}
\item For any $W\in \mathsf{L}^2(\mathbb{R})$ we have $\mathcal{A}_{\eta}\BC W\EC\in \mathsf{L}^2(\mathbb{R})\cap \mathsf{L}^{\infty}(\mathbb{R})$ with 
\begin{align*}
||\mathcal{A}_{\eta}W||_{\infty}\leq \eta^{-1/2}||W||_{2},\qquad ||\mathcal{A}_{\eta}W||_{2}\leq ||W||_{2}.
\end{align*}
Moreover, \BMHC  $\mathcal{A}_{\eta}W$ \EMHC admits a weak derivative with $||(\mathcal{A}_{\eta}W)'||_{2}\leq \eta^{-1/2}||W||_{2}$.
\item For any $W\in \mathsf{L}^{\infty}(\mathbb{R})$, we have $||\mathcal{A}_{\eta}W||_{\infty}\leq ||W||_{\infty}$.
\item $\mathcal{A}_{\eta}$ \BMHC respects the even-odd parity and the non-negativity of functions.\EMHC
\item $\mathcal{A}_{\eta}$ diagonalizes in Fourier space and corresponds to the symbol function \BMHC from \eqref{Eqn:SymbolA}.\EMHC
\item $\mathcal{A}_{\eta}$ is self-adjoint in the $\mathsf{L}^2$-sense.
\item The operators $\mathcal{A}_{\eta_1}$ and $ \mathcal{A}_{\eta_2}$ commute with each other for any $\eta_1, \eta_2>0$.
\item There exists a constant $C>0$, which does not depend on $\eta$, such that the estimates 
\begin{align*}
||\mathcal{A}_{\eta}W-W||_2\leq C\eta^2||W''||_2,\qquad ||\mathcal{A}_{\eta}W-W||_{\infty}\leq C\eta^2||W''||_{\infty}
\end{align*}
and
\begin{align*}
||\mathcal{A}_{\eta}W-W-\frac{\eta^2}{24}W''||_2\leq C\eta^4||W''''||_2,\qquad ||\mathcal{A}_{\eta}W-W-\frac{\eta^2}{24}W''||_{\infty}\leq C\eta^4||W''''||_{\infty}
\end{align*}
hold for any sufficiently regular $W$. In particular, we have
\begin{align}\label{rigorexp}
\mathcal{A}_{\eta}W\rightarrow W\quad \text{strongly in}\quad \mathsf{L}^2(\mathbb{R})
\end{align}
for any $W\in \mathsf{L}^2(\mathbb{R})$.
\end{enumerate}
\end{Lem}
\begin{proof}
\BMHC Equation \eqref{Def:OperatorAeta} can be written as $\mathcal{A}_\eta W=\chi_\eta\ast W$, where $
\chi_\eta$ is the normalized indicator function of an interval and satisfies $\|\chi_\eta\|_1=1$. Moreover, the function in \eqref{Eqn:SymbolA} is just the Fourier transform of $\chi_\eta$.
All assertions thus follow from standard arguments concerning convolutions and their Fourier representations, see \cite[Lemma 2.5]{Her10}, \cite[Lemma 5]{HML15} and also \cite{Che13,Che17} for the details.\EMHC
\end{proof}
\subsection{Linear main part} 
Since for traveling waves the functions $F_i^m$ represent the partial derivatives of the effective potentials, see for instance \eqref{motion},
a \BMHC first \EMHC natural condition can be formulated as follows.
 \begin{Ass}[\BMHC linear and quadratic coefficients \EMHC]\label{Ass sym}
The \BMHC Taylor \EMHC coefficients of $F_i^m$, see \eqref{Ide1}, satisfy
\begin{align*}
\alpha_{1,2}^m=\alpha_{2,1}^m\,,\qquad \beta_{1,22}^m=\beta_{2,12}^m\,,\qquad  \beta_{1,12}^m=\beta_{2,11}^m,
\end{align*}
for all $m=1,...,M$.
 \end{Ass}
As already explained in the introduction, our strategy is to collect the linear and the quadratic terms from \eqref{equep1} into two operators $\mathcal{B}_\epsilon$ and $\mathcal{Q}_\epsilon$, respectively, as in \eqref{BW}. A first natural choice for the linear part would be 
\begin{align}
\label{Bcan}
\mathcal{B}_{\epsilon}^{\text{can}}:=\frac{1}{\epsilon^2}\begin{pmatrix}
\sigma_{\epsilon}-\sum_{m=1}^Mk_m^2\alpha_{1,1}^m\mathcal{A}_{k_m\epsilon}^2& -\sum_{m=1}^Mk_m^2\alpha_{1,2}^m\mathcal{A}_{k_m\epsilon}^2 \\\\
 -\sum_{m=1}^Mk_m^2\alpha_{2,1}^m\mathcal{A}_{k_m\epsilon}^2 &  \sigma_{\epsilon}-\sum_{m=1}^Mk_m^2\alpha_{2,2}^m\mathcal{A}_{k_m\epsilon}^2
\end{pmatrix},
 \end{align} 
but this operator does not converge as $\epsilon\to0$. In fact, \BMHC in view of \EMHC \eqref{speed} the dominant part is given by
\begin{align*}
\mathcal{B}_{\epsilon}^{\text{sing}}:=\frac{1}{\epsilon^2}\begin{pmatrix}
\sigma_{0}-c_1& -c_2 \\
 -c_2 &  \sigma_{0}-c_3
\end{pmatrix}
 \end{align*} 
 with
\begin{align}
\label{Eqn:DefLinCoeff}
c_1:=\sum_{m=1}^{M}k_m^2\alpha_{1,1}^m,\qquad c_2:=\sum_{m=1}^{M}k_m^2 \alpha_{1,2}^m=\sum_{m=1}^{M}k_m^2 \alpha_{2,1}^m,\qquad c_3:=\sum_{m=1}^{M}k_m^2 \alpha_{2,2}^m
\end{align}
and 
diverges as $\epsilon\to0$. In what follows we therefore define the operator $\mathcal{B}_{\epsilon}$ in a slightly different \BMHC way. Moreover, we choose $\sigma_0$ such that $\mathcal{B}_{\epsilon}^{\text{sing}}$ has a nontrivial kernel because otherwise any solution to the leading order equation \eqref{BWLim} would vanish identically. The latter condition reads \EMHC 
\begin{equation}\label{eqsigma}
(c_1-\sigma_0)(c_3-\sigma_0)=c_2^2
\end{equation}
and provides two possible solution branches. However, it turns out that only the larger value of $\sigma_0$ is admissible for our asymptotic analysis, see the remarks at the beginning of \S\ref{sect:app}, and thus we set
\begin{equation}\label{defsig}
\sigma_0=\frac{(c_1+c_3)+\sqrt{(c_1-c_3)^2+c_2^2}}{2}.
\end{equation}
Using this particular choice of $\sigma_0$ we now define
\begin{align}
\label{Eqn:DefBEps}
\mathcal{B}_{\epsilon}:=\mathcal{T}_{\epsilon}\cdot\mathcal{B}_{\epsilon}^{\text{can}}
\end{align}
with
\begin{align}\label{deflam}
\mathcal{T}_{\epsilon}:=\begin{pmatrix}
1 & \lambda\\
0 & \epsilon^2
\end{pmatrix}, \qquad \lambda:=\frac{c_2}{\sigma_0-c_3}=\frac{\sigma_0-c_1}{c_2},
\end{align}
\BMHC which implies \EMHC that the corresponding \BMHC leading order \EMHC  part is given by the \BMHC non-invertible matrix \EMHC
\begin{align*}
\mathcal{T}_{\epsilon}\cdot\mathcal{B}_{\epsilon}^{\text{sing}}=\begin{pmatrix}
0 & 0\\
-c_2 &  \sigma_{0}-c_3
\end{pmatrix}.
\end{align*}
Using this definition of $\mathcal{B}_\epsilon$, the nonlinear integral equation \eqref{equep1} can \BMHC be \EMHC reformulated as \eqref{BW}, where the quadratic and higher terms correspond to operators $\mathcal{Q}_\epsilon$ and $\mathcal{P}_\epsilon$ with components
\begin{align}
\label{Eqn:DefQeps}
\begin{split}
(\mathcal{Q}_{\epsilon}[W])_{1}&:=\frac{1}{2}\sum_{m=1}^{M}\Big(k_m^3(\beta_{1,11}^m+\lambda\beta_{2,11}^m)\mathcal{A}_{k_m\epsilon}(\mathcal{A}_{k_m\epsilon}W_1)^2+2k_m^3(\beta_{1,12}^m+\lambda\beta_{2,12})\\
&\qquad\qquad + \mathcal{A}_{k_m\epsilon}\Big((\mathcal{A}_{k_m\epsilon}W_1)(\mathcal{A}_{k_m\epsilon}W_2)\Big)+k_m^3(\beta_{1,22}^m+\lambda\beta_{2,22}^m)\mathcal{A}_{k_m\epsilon}(\mathcal{A}_{k_m\epsilon}W_2)^2\Big)
\\
(\mathcal{Q}_{\epsilon}[W])_{2}&:=\frac{\epsilon^2}{2}\sum_{m=1}^{M}\Big(k_m^3\beta_{2,11}^m\mathcal{A}_{k_m\epsilon}(\mathcal{A}_{k_m\epsilon}W_1)^2+2k_m^3\beta_{2,12}^m\mathcal{A}_{k_m\epsilon}(\mathcal{A}_{k_m\epsilon}W_1)(\mathcal{A}_{k_m\epsilon}W_2)\\
&\qquad\qquad  +k_m^3\beta_{2,22}^m\mathcal{A}_{k_m\epsilon}(\mathcal{A}_{k_m\epsilon}W_2)^2\Big)
\end{split}
\end{align}
and 
\begin{align}
\notag%
(\mathcal{P}_{\epsilon}[W])_{1}&:=\frac{1}{\epsilon^6}\sum_{m=1}^{M}k_m\Big(\Psi_1^m(\epsilon^2k_m\mathcal{A}_{k_m\epsilon}W_{1},\epsilon^2k_m\mathcal{A}_{k_m\epsilon}W_{ 2})+\lambda  \Psi_2^m(\epsilon^2k_m\mathcal{A}_{k_m\epsilon}W_{ 1},\epsilon^2k_m\mathcal{A}_{k_m\epsilon}W_{2})\Big)\\
\label{Eqn:DefPeps}
(\mathcal{P}_{\epsilon}[W])_{2}&:=\frac{1}{\epsilon^4}\sum_{m=1}^{M}k_m\Psi_2^m(\epsilon^2k_m\mathcal{A}_{k_m\epsilon}W_{ 1},\epsilon^2k_m\mathcal{A}_{k_m\epsilon}W_ {2}).
\end{align}
\subsection{Leading order problem and KdV wave}
\BC By straightforward \EC\BMHC computations we verify the formal expansion
\begin{align*}
\mathcal{B}_{\epsilon}W:=\mathcal{B}_{0}W+O(\epsilon^2)\,,\qquad Q_\epsilon[W] =Q_0[W] 
+O(\epsilon^2)
\end{align*}
with leading order terms
\begin{align}
\label{Eqn:LeadingOrder1}
\mathcal{B}_0W:=
\begin{pmatrix}
W_1-(a_1+\lambda a_2)W_1^{\prime\prime}&&\lambda W_2-(b_2+\lambda b_1)W_2^{\prime\prime}\\
-c_2W_1&&(\sigma_0-c_3)W_2
\end{pmatrix}
\end{align}
and
\begin{align}
\label{Eqn:LeadingOrder2}
Q_0[W] =
\begin{pmatrix}
(a_3+\lambda  b_4)W_1^2 + (a_5+\lambda  b_5)W_1W_2+(a_4+\lambda  b_3)W_2^2\\0 
\\
\end{pmatrix} .
\end{align}
Here, the constants are given by 
\begin{align*}
a_1:=\sum_{m=1}^{M}\frac{k_m^4}{12} \alpha_{1,1}^m,\qquad a_2:=\sum_{m=1}^{M}\frac{k_m^4}{12} \alpha_{1,2}^m,\qquad 
b_1:=\sum_{m=1}^{M}\frac{k_m^4}{12} \alpha_{2,2}^m,\qquad b_2:=\sum_{m=1}^{M}\frac{k_m^4}{12} \alpha_{2,1}^m \qquad
\end{align*}
as well as
\begin{align*}
a_3:= \sum_{m=1}^{M}\frac{k_m^3}{2} \beta_{1,11}^m ,\qquad  a_4:= \sum_{m=1}^{M}\frac{k_m^3}{2} \beta_{1,22}^m, \qquad  a_5:= \sum_{m=1}^{M}k_m^3 \beta_{1,12}^m
\end{align*}
and
\begin{align*}
b_3:= \sum_{m=1}^{M}\frac{k_m^3}{2} \beta_{2,22}^m, \qquad  b_4:= \sum_{m=1}^{M}\frac{k_m^3}{2} \beta_{2,11}^m, \qquad  b_5:= \sum_{m=1}^{M}k_m^3 \beta_{2,12}^m.
\end{align*}
The first and the second component of the leading order equation \eqref{BWLim} can therefore be written as \eqref{Ws1} and \eqref{Ws}, respectively, where 
\begin{align}\label{defd}
d_1:=\frac{(1+\lambda^2)}{(a_1+\lambda a_2)+\lambda(b_2+\lambda b_1)}\,,\qquad  d_2:=\frac{(a_3+\lambda^2 a_4+\lambda a_5)+\lambda(\lambda b_5+\lambda^2 b_3+b_4)}{(a_1+\lambda a_2)+\lambda(b_2+\lambda b_1)}
\end{align}
denote the constants in \eqref{Ws}. This ODE for the KdV wave is a planar Hamiltonian system with potential energy $E[W]:=\frac{1}{3}d_2W^3-\frac{1}{2}d_1W^2$, see Figure \ref{Fig.ODE}, and its coefficients $d_1$, $d_2$ are completely determined by the linear and quadratic terms of $F_i^m$. A simple phase plane analysis reveals that \eqref{Ws} admits a non-trivial homoclinic orbit \EMHC \BC with unstable equilibrium at $0$ \EC \BMHC if and only if $d_1>0$ and $d_2\neq0$. Moreover, we readily verify that the corresponding even solution $W_*\in\mathsf{L}_{\rm{even}}^2(\mathbb{R})$ is given by 
\begin{align}\label{formelW}
W_*(\xi)=\frac{3d_1}{2d_2} {\sech}^2\big(\tfrac{1}{2}\sqrt{d_1}\xi\big)
\end{align}
and corresponds to the shape parameters
\begin{align}
\label{Eqn:ShapePrms}
p_1:=\max\limits_{\xi\in \mathbb{R}}W_*(\xi)=\frac{3d_1}{2d_2},\,\qquad
p_2:=\max\limits_{\xi\in \mathbb{R}}W'_*(\xi)=\sqrt{\frac{d_1^3}{3d_2^2}}.
\end{align}
\begin{figure}[ht!] %
\begin{center} %
   \includegraphics[width=0.75\textwidth]{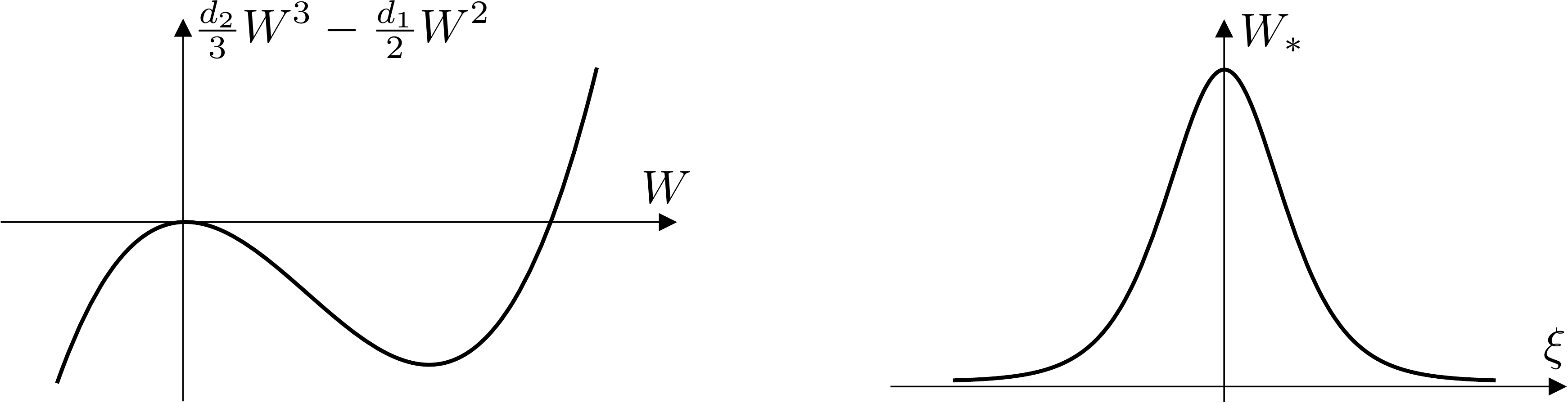} %
\end{center} %
\caption{ %
\textit{Left panel}: Graph of the potential energy \BMHC for the limit ODE \eqref{Ws}. \textit{Right panel}: The unique homoclinic solution in \EMHC  $\mathsf{L}_{\rm{even}}^2(\mathbb{R})$, which corresponds to the region between the two zeros of $E[W]:=\frac{d_2}{3}W^3-\frac{d_1}{2}W^2$. 
} %
\label{Fig.ODE} %
\end{figure}%
\EMHC
\BC In this paper we always assume that the leading order problem \eqref{Ws} admits the homoclinic solution \eqref{formelW}. To avoid technicalities also suppose that $c_2$ does not vanish and restrict our considerations to the case $\sigma_0>0$, which provides a uniform lower bound for the speed $\sqrt{\sigma_\varepsilon}$ of the solitary lattice waves according to \eqref{speed}. The one-dimensional analogue to the latter condition is a nonvanishing (and real-valued) sound speed, see \cite{HML15}. In summary, we rely on the following assumption.
\EC \begin{Ass}[\BC genericity \EC and well-definedness of KdV waves]
The constants defined in \BC \eqref{Eqn:DefLinCoeff}, \eqref{defsig}, and \eqref{defd} \EC satisfy
\begin{enumerate}
\label{Ass2}
\item \BC $ \sigma_0>0$ and $c_2\neq0$, \EC
\item \BC $d_1>0$ and $d_2\neq0$.\EC
 \end{enumerate}
\end{Ass} 
\subsection{Fixed-point equation for the corrector  \texorpdfstring{$V_\varepsilon$}{V}}\label{sect:reform4}
Having determined $\sigma_0$ and $W_0$, we turn to the discussion of the equation for the corrector term $V_{\epsilon}$. \BMHC Substituting the ansatz \eqref{decompos} \EMHC into \eqref{BW} we obtain
\begin{align}
\label{vorfix} %
\mathcal{L}_{\epsilon}V_{\epsilon}:=(\mathcal{B}_{\epsilon}-\mathcal{M}_{\epsilon})V_{\epsilon}=\epsilon^2\big(\mathcal{Q}_{\epsilon}[V_{\epsilon}]+\mathcal{N}_{\epsilon}[W_{0};V_{\epsilon}]\big)+\big(R_{\epsilon}[W_{0}]+\mathcal{P}_{\epsilon}[W_{0}]\big),
\end{align}
where
\begin{align*}
R_{\epsilon}[W_{0}]:=\frac{\mathcal{Q}_{\epsilon}[W_{0}]-\mathcal{B}_{\epsilon}W_{0}}{\epsilon^2},\qquad \mathcal{N}_{\epsilon}[W_{0};V_\epsilon]:=\frac{\mathcal{P}_{\epsilon}[W_{0}+\epsilon^2V_\epsilon]-\mathcal{P}_{\epsilon}[W_{0}]}{\epsilon^2}.
\end{align*}
\BMHC Moreover, the linear operator
$\mathcal{M}_{\epsilon}$ depends on the KdV wave $W_*$ from \eqref{formelW} and is given by\EMHC 
\begin{align}\label{defM}
\mathcal{M}_{\epsilon}:=\begin{pmatrix}
\sum_{m=1}^{M}\eta_{1,1}^m\mathcal{A}_{k_m\epsilon}\Big((\mathcal{A}_{k_m\epsilon}W_*)\mathcal{A}_{k_m\epsilon}\Big) & \sum_{m=1}^{M}\eta_{1,2}^m\mathcal{A}_{k_m\epsilon}\Big((\mathcal{A}_{k_m\epsilon}W_*)\mathcal{A}_{k_m\epsilon}\Big)\\
 \epsilon^2\sum_{m=1}^{M}\eta_{2,1}^m\mathcal{A}_{k_m\epsilon}\Big((\mathcal{A}_{k_m\epsilon}W_*)\mathcal{A}_{k_m\epsilon}\Big)& \epsilon^2\sum_{m=1}^{M}\eta_{2,2}^m\mathcal{A}_{k_m\epsilon}\Big((\mathcal{A}_{k_m\epsilon}W_*)\mathcal{A}_{k_m\epsilon}\Big)
 \end{pmatrix}
  \end{align}
\BMHC with  constants \EMHC
\begin{align*}
\eta_{1,1}^m&:=k_m^3\big((\beta_{1,11}^m+\lambda\beta_{1,12}^m)+\lambda( \lambda \beta_{2,12}^m+\beta_{2,11}^m)\big),\\
\eta_{1,2}^m&:=k_m^3\big((\lambda \beta_{1,22}^m+\beta_{1,12}^m)+\lambda(\lambda \beta_{2,22}^m+\beta_{2,12}^m)\big),\\
\eta_{2,1}^m&:=k_m^3( \lambda \beta_{2,12}^m+\beta_{2,11}^m),\\
\eta_{2,2}^m&:=k_m^3(\lambda \beta_{2,22}^m+\beta_{2,12}^m).
\end{align*}
The uniform invertibility of $\mathcal{L}_{\epsilon}$ will be \BMHC crucial in our \EMHC existence proof, for it allows us to \BMHC reformulate \eqref{vorfix} as the fixed point equation \eqref{Eqn:FixedPoint}, where the operator
\begin{align}\label{endeq}
\mathcal{F}_{\epsilon}[V_{\epsilon}]:=\epsilon^2\mathcal{L}_{\epsilon}^{-1}\Big(\mathcal{Q}_{\epsilon}[V_{\epsilon}]+\mathcal{N}_{\epsilon}[W_{0};V_{\epsilon}]\Big)+\mathcal{L}_{\epsilon}^{-1}\Big(R_{\epsilon}[W_{0}]+\mathcal{P}_{\epsilon}[W_{0}]\Big),
\end{align}
turns out to be contractive in some bounded set provided that $\epsilon$ is sufficiently small. Before we discuss the construction and the properties of $\mathcal{F}_\epsilon$ in greater detail, we formulate a further assumption which guarantees that the higher order terms are negligible compared to the linear and quadratic terms. In the \EMHC simplest case we have $\Psi_i^m\equiv0$.
\begin{Ass}[regularity of higher order terms]\label{Ass1}
The remainder terms $\Psi_i^m$ \BMHC in the Taylor expansion \eqref{Ide1} \EMHC satisfy
\begin{equation*}
\Psi_i^m(0,0)=0
\end{equation*}
and there exists constants \BMHC $\gamma_i^m$ such that \EMHC
\begin{equation*}
\big|\Psi_i^m(x_1,x_2)-\Psi_i^m(y_1,y_2)\big|\leq \gamma_i^m\big(x_1^2+x_2^2+y_1^2+y_2^2\big)\big(|x_1-y_1|+|x_2-y_2|\big)
\end{equation*}
for all $(x_1,x_2), (y_1,y_2)\in \mathbb{R}^2$, $i=1,2$ and $m=1,...,M$.
\end{Ass}
\subsection{Inverting the operator \texorpdfstring{$\mathcal{B}_\epsilon$}{B} }\label{sect:invB}
\BMHC A particular challenge in the existence proof for $\mathcal{L}_\epsilon^{-1}$ is to invert the operator $\mathcal{B}_\epsilon$ \BC from \eqref{Eqn:DefBEps} \EC \BMHC for small values of $\epsilon>0$.
Since all operators $\mathcal{A}_{k_m\epsilon}$ commute with each another, it is 
 quite natural to consider the `determinant' of $\mathcal{B}_{\epsilon}$ defined by 
\begin{equation}
\label{Eqn:DefDetBEps}
\det{\mathcal{B}_{\epsilon}}:=\mathcal{B}_{\epsilon,11}\mathcal{B}_{\epsilon,22}-\mathcal{B}_{\epsilon,12}\mathcal{B}_{\epsilon,21}.
\end{equation}
In fact, if this operator is invertible on $\mathsf{L}^2(\mathbb{R})$, a straightforward calculation reveals that the inverse of $\mathcal{B}_{\epsilon}$ can be written as
\begin{align}
\label{Eqn:InvFormula}
\mathcal{B}_{\epsilon}^{-1}&:=(\det{\mathcal{B}_{\epsilon}})^{-1}\begin{pmatrix}
+\mathcal{B}_{\epsilon, 22}& -\mathcal{B}_{\epsilon, 12}\\
 -\mathcal{B}_{\epsilon, 21}&  +\mathcal{B}_{\epsilon, 11}
  \end{pmatrix}.
  \end{align}
In order to prove the uniform invertibility of $\det\mathcal{B}_{\epsilon}$, we introduce two auxiliary functions
\begin{equation}\label{connect}
S_{k}(z):=1-{\sinc}^2(kz/2)\,,\qquad s_{k,\epsilon}(z):=\frac{1}{\epsilon^2}S_{k}(\epsilon z)=\frac{1-{\sinc}^2(k\epsilon z/2)}{\epsilon^2},
\end{equation}
for $k\in\mathbb{R}$ and $z\in \mathbb{R}$, see Figure \ref{Fig:Sinc} and \eqref{Eqn:SymbolA}. Moreover, we rely on the following assumption, which will be checked numerically in \S\ref{sect:app} for \BMHC several lattice geometries. \EMHC
\begin{Ass}[conditions for the inversion of $\mathcal{B}_{\epsilon}$]\label{Ass inverse} 
There exists a constant \BMHC $\delta_0>0$, \EMHC such that
\begin{align*}
T(z)&:=\big(2\sigma_0-(c_1+c_3)\big)\cdot \Bigg( %
\sum_{m=1}^{M}k_m^2\Big((\sigma_0-c_3)\alpha_{1,1}^m+(\sigma_0-c_1)\alpha_{2,2}^m+2c_2\alpha_{1,2}^mS_{k_m}(z)\Big)+
\\&+\Big(\sum_{m=1}^{M}k_m^2\alpha_{1,1}^m S_{k_m}(z)\Big)\Big(\sum_{m=1}^{M}k_m^2\alpha_{2,2}^m S_{k_m}(z)\Big)- \Big(\sum_{m=1}^{M} k_m^2\alpha_{1,2}^mS_{k_m}(z)\Big)^2
\Bigg)\geq \delta_0\big(\min\{|z|,2\}\big)^2
\end{align*}
holds for all $z\in \mathbb{R}$. 
\end{Ass}
\BMHC The function $T$ is independent of $\epsilon$ and can be locally expanded as $T(z)=\tau z^2+O(z^4)$ near $z=0$, where $\tau:=\frac{1}{24}\sum_{m=1}^{M}k_m^4\big((\sigma_0-c_3)\alpha_{1,1}^m+(\sigma_0-c_1)\alpha_{2,2}^m+2c_2\alpha_{1,2}^m)\big)$. If $\tau>0$, then the required condition is satisfied in the vicinity of $z=0$ for any $0<\delta_0<\tau$.\EMHC
 \begin{figure}[ht!]
\begin{center}
   \includegraphics[width=0.8\textwidth]{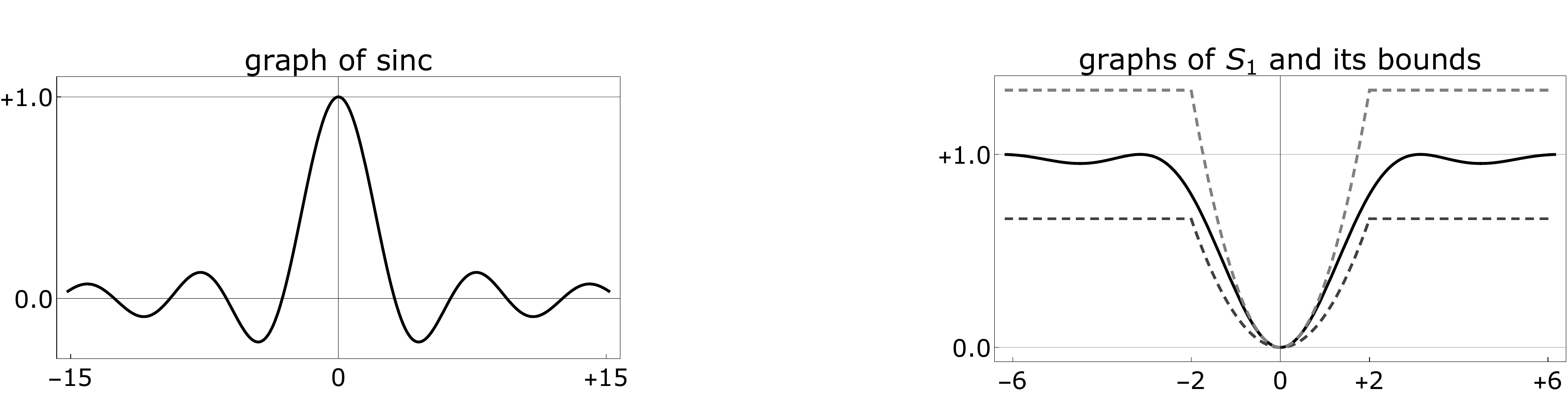}
\end{center} 
\caption{\textit{Left panel}: Graph of the sinc function. \textit{Right panel}: Lower bound $\tfrac{1}{6}\min\{|z|,2\}^2$ (dashed) and upper bound $\tfrac{1}{3}\min\{|z|,2\}^2$ (dashed) for $S_1=1-\sinc^2$ (solid). }
\label{Fig:Sinc}
\end{figure}
  \begin{Lem}[\BMHC existence of $\mathcal{B}_\epsilon^{-1}$\EMHC] \label{estB}
\BMHC Assumption~\ref{Ass inverse} implies for \EMHC \BC all \EC \BMHC sufficiently small $\epsilon>0$ 
that the symbol function $\widehat{\det(\mathcal{B}_{\epsilon})}$ of the operator from \EMHC \eqref{Eqn:DefDetBEps} satisfies
\begin{align*}
\widehat{\det(\mathcal{B}_{\epsilon})}(z)\geq\BC 2\sigma_0-(c_1+c_3)\EC+\frac{\delta}{\epsilon^2}(\min\{|\epsilon z|,2\})^2
  \end{align*}
\BMHC for all $z\in \mathbb{R}$ and some constant $\delta>0$ independent of $\epsilon$, \EMHC
\BC where $2\sigma_0-(c_1+c_3)>0$ holds according to \eqref{defsig}. In particular,  $\det\mathcal{B}_\epsilon$ is uniformly invertible on $\mathsf{L}^2(\mathbb{R})$. \EC
\end{Lem}
\begin{proof}
We set
\begin{align*}
T_1(z):=\sum_{m=1}^{M}k_m^2(\alpha_{1,1}^m+\alpha_{2,2}^m)S_{k_m}(z)
\end{align*}
\BMHC and notice that for any $m=1,..,M$ we have $|S_{k_m}(z)|\leq C_{m}(\min\{|z|,2\})^2$
for some constant $C_m$. This implies \EMHC
\begin{align}\label{estT1}
|T_1(z)|\leq \delta_1\big(\min\{|z|,2\}\big)^2,
\end{align}
for all $z\in \mathbb{R}$ and a sufficiently large \BMHC constant \EMHC $\delta_1>0$. The Fourier \BMHC symbols \EMHC of the components of $\mathcal{B}_{\epsilon}^{\text{can}}$ are given by
\begin{align*}
\widehat{\mathcal{B}}^{\text{can}}_{\epsilon,11}&=\frac{1}{\epsilon^2}\Big((\sigma_0-c_1)+\epsilon^2(1+\sum_{m=1}^{M}k_m^2\alpha_{1,1}^ms_{k_m,\epsilon})\Big),\\ \widehat{\mathcal{B}}^{\text{can}}_{\epsilon,22}&=\frac{1}{\epsilon^2}\Big((\sigma_0-c_3)+\epsilon^2(1+\sum_{m=1}^{M}k_m^2\alpha_{2,2}^ms_{k_m,\epsilon})\Big),\\
\widehat{\mathcal{B}}^{\text{can}}_{\epsilon,12}=\widehat{\mathcal{B}}^{\text{can}}_{\epsilon,21}&=\frac{1}{\epsilon^2}\Big(-c_2+\epsilon^2 \sum_{m=1}^{M}k_m^2\alpha_{1,2}^ms_{k_m,\epsilon}\Big).
 \end{align*}
 Now we compute the Fourier symbol function $\widehat{\det(\mathcal{B}_{\epsilon})}$ of the determinant $\det(\mathcal{B}_{\epsilon})$ by substituting these expressions. \BMHC More precisely, 
combining the above formulas with \eqref{eqsigma} and \eqref{connect} we get\EMHC
 \begin{align*}
\widehat{\det(\mathcal{B}_{\epsilon})}=&\; \widehat{\det(\mathcal{T}_{\epsilon}\cdot\mathcal{B}_{\epsilon}^{\text{can}})}=\epsilon^2\big(\widehat{\mathcal{B}}^{\text{can}}_{\epsilon,11}\widehat{\mathcal{B}}^{\text{can}}_{\epsilon,22}-\widehat{\mathcal{B}}^{\text{can}}_{\epsilon,12}\widehat{\mathcal{B}}^{\text{can}}_{\epsilon,21}\big)
%
%
%
\\
=&\;\big(2\sigma_0-(c_1+c_3)+\epsilon^2\big)+\frac{T(\epsilon z)}{(2\sigma_0-(c_1+c_3))\epsilon^2}+T_1(\epsilon z),
\end{align*} 
\BC  and thanks to \eqref{defsig}, Assumption \ref{Ass inverse}, and \eqref{estT1} we show that \EC
\begin{align*}
\frac{T(\epsilon z)}{(2\sigma_0-(c_1+c_3))\epsilon^2}+T_1(\epsilon z)\geq \Bigg(\frac{\delta_0}{(2\sigma_0-(c_1+c_3))\epsilon^2}-\delta_1\Bigg) \big(\min\{|\epsilon z|,2\}\big)^2\geq \frac{\delta}{\epsilon^2}\big(\min\{|\epsilon z|,2\}\big)^2
\end{align*}
\BMHC holds for all sufficiently small $\epsilon>0$. This implies
\begin{align*}
\widehat{\det(\mathcal{B}_{\epsilon})} \geq 2\sigma_0-(c_1+c_3)+\frac{\delta}{\epsilon^2}(\min\{|\epsilon z|,2\})^2\geq 2\sigma_0-(c_1+c_3)>0
\end{align*} 
for all $z\in \mathbb{R}$, and we conclude that 
$\det(\mathcal{B}_{\epsilon})$ is uniformly invertible as its Fourier symbol is bounded from below by a positive constant independent of $\epsilon$. \EMHC 
\end{proof}
\BMHC In order to characterize the asymptotic properties of $\mathcal{B}_\epsilon$ we introduce a \EMHC cut-off operator 
\begin{align}
\label{piop1}
\Pi_{\epsilon}: \mathsf{L}^2(\mathbb{R})\rightarrow \mathsf{L}^2(\mathbb{R})
\end{align}
by \BMHC defining the \EMHC corresponding symbol function in Fourier space via
\begin{align}\label{piop2}
\pi_{\epsilon}(z):=\begin{cases}
 1 & \text{for}\ |z|\leq \frac{2}{\epsilon},\\
 0 &  \text{otherwise}.
 \end{cases}
\end{align}
\BMHC This gives rise to the following estimates for $\mathcal{B}_{\epsilon}^{-1}$, \EMHC which below remove the main difficulty in the inversion of $\mathcal{L}_{\epsilon}$. 
  \begin{Lem}[\BMHC properties of $\mathcal{B}_\epsilon^{-1}$\EMHC]\label{Lemma 6}
\BMHC There exists a constant $C$ independent of $\epsilon$ such that \EMHC
\begin{align*}
||\Pi_{\epsilon}(\mathcal{B}_{\epsilon}^{-1})_{i1}G||_{2,2}+\epsilon^{-2}||(\id-\Pi_{\epsilon})(\mathcal{B}_{\epsilon}^{-1})_{i1}G||_{2}\leq C||G||_2\,,\qquad
||(\mathcal{B}_{\epsilon}^{-1})_{i2}G||_2\leq C||G||_2
  \end{align*}
\BMHC hold for any $G\in \mathsf{L}^2(\mathbb{R})$ and all sufficiently small $\epsilon>0$.\EMHC
\end{Lem}
\begin{proof}
\BMHC Notice that for each $k>0$ there exists a constant $C_k>0$ such that the estimate $S_k(z)\leq C_k(\min\{|z|,2\})^2$ for all $z\in \mathbb{R}$, see Figure \ref{Fig:Sinc} and \eqref{connect}. Combining this with the Fourier representation of $\mathcal{A}_{k_m\epsilon}$ from Lemma \ref{Prop A}, Formula \eqref{Eqn:InvFormula}, and Lemma \ref{estB} we get
    \begin{align*}
    \big|\widehat{({\mathcal{B}}_\epsilon^{-1})_{11}}\big|=\big|\widehat{\det(\mathcal{B}_\epsilon)}^{-1}\widehat{\mathcal{B}}_{\epsilon,22}\big|&\leq \frac{\big|\sigma_0-c_3\big|+1+\sum_{m=1}^{M}k_m^2|\alpha_{2,2}^m|}{2\sigma_0-(c_1+c_3)+\frac{\delta}{\epsilon^2}(\min\{|\epsilon z|,2\})^2}\leq \BMHC \begin{cases}
 \frac{C}{1+z^2}& \text{for}\ |z|\leq \frac{2}{\epsilon },\\
 C\epsilon^2&  \text{otherwise,}
 \end{cases}
  \end{align*}
 and similarly we find 
\begin{align*}
    \big|\widehat{({\mathcal{B}}_\epsilon^{-1})_{21}}\big|=\big|\widehat{\det(\mathcal{B}_\epsilon)}^{-1}\widehat{\mathcal{B}}_{\epsilon, 21}\big|\leq \frac{|c_2|+\sum_{m=1}^{M}k_m^2|\alpha_{1,2}^m|}{2\sigma_0-(c_1+c_3)+\frac{\delta}{\epsilon^2}(\min\{|\epsilon z|,2\})^2}\leq\BMHC \begin{cases}
 \frac{C}{1+z^2}&\text{for}\ |z|\leq \frac{2}{\epsilon },\\
 C\epsilon^2& \text{otherwise},\EMHC
 \end{cases}
  \end{align*}
where we did not write the argument $z$ on the left hand side to ease the notation. Theses identities combined with \eqref{piop1} and \eqref{piop2} imply \EMHC
\begin{align*}
||\Pi_{\epsilon}(\mathcal{B}_{\epsilon}^{-1})_{i1}G||_{2,2}^2&\leq \int\limits_{|z|\leq \frac{2}{\epsilon}}(1+z^2+z^4) \Big(\widehat{({\mathcal{B}}_\epsilon^{-1})_{i1}G}\Big)^2 dz= \int\limits_{|z|\leq \frac{2}{\epsilon}}(1+z^2+z^4) \widehat{({\mathcal{B}}_\epsilon^{-1})_{i1}}^2\widehat{G}^2 dz\\
&\leq  C\int\limits_{|z|\leq \frac{2}{\epsilon}}\frac{1+z^2+z^4}{1+2z^2+z^4}\widehat{G}^2 dz\leq C||G||_2^2
\end{align*}
\BMHC as well as \EMHC
\begin{align*}
||(\id-\Pi_{\epsilon})(\mathcal{B}_{\epsilon}^{-1})_{i1}G||_{2,2}^2&\leq \int\limits_{|z|\geq \frac{2}{\epsilon}}\Big(\widehat{(\mathcal{B}_\epsilon^{-1})_{i1}G}\Big)^2 dz= \int\limits_{|z|\geq \frac{2}{\epsilon}}\widehat{({\mathcal{B}}_\epsilon^{-1})_{i1}}^2\widehat{G}^2 dz\\
&\leq  C\epsilon^4\int\limits_{|z|\geq \frac{2}{\epsilon}}\widehat{G}^2 dz\leq C\epsilon^4||G||_2^2
\end{align*}
for all $G\in \mathsf{L}^2(\mathbb{R})$ and $i=1,2$. \BMHC In particular, we showed the first claimed estimate. Towards the second one, we observe that
 \begin{align*}
    \big|\widehat{({\mathcal{B}}_\epsilon^{-1})_{12}}\big|=\big|\widehat{\det(\mathcal{B}_\epsilon)}^{-1}\widehat{\mathcal{B}}_{\epsilon,12}\big|&\leq \frac{\lambda+\Big(\sum_{m=1}^{M}k_m^2(|\alpha_{2,1}^m|+\lambda|\alpha_{2,2}^m|)\Big)\frac{\kappa}{\epsilon^2}(\min\{|\epsilon z|,2\})^2}{2\sigma_0-(c_1+c_3)+\frac{\delta}{\epsilon^2}(\min\{|\epsilon z|,2\})^2}\leq C
  \end{align*}
 and 
 \begin{align*}
    \big|\widehat{({\mathcal{B}}_\epsilon^{-1})_{12}}\big|=\big|\widehat{\det(\mathcal{B}_\epsilon)}^{-1}\widehat{\mathcal{B}}_{\epsilon,12}\big|&\leq \frac{\lambda+\Big(\sum_{m=1}^{M}k_m^2(|\alpha_{2,1}^m|+\lambda|\alpha_{2,2}^m|)\Big)\frac{\kappa}{\epsilon^2}(\min\{|\epsilon z|,2\})^2}{2\sigma_0-(c_1+c_3)+\frac{\delta}{\epsilon^2}(\min\{|\epsilon z|,2\})^2}\leq C
  \end{align*}
   hold for all $z\in \mathbb{R}$, and estimate\EMHC 
\begin{align*}
||(\mathcal{B}_{\epsilon}^{-1})_{i2}G||_{2,2}^2\leq \int_{\mathbb{R}} \Big(\widehat{({\mathcal{B}}_\epsilon^{-1})_{i2}G}\Big)^2 dz= \int_{\mathbb{R}} \widehat{({\mathcal{B}}_\epsilon^{-1})_{i2}}^2\widehat{G}^2 dz\leq  C\int_{\mathbb{R}}\widehat{G}^2 dz\leq C||G||_2^2
\end{align*}
for all $G\in \mathsf{L}^2(\mathbb{R})$.
\end{proof}
\subsection{Further auxiliary results}\label{sect:Aux2}
\BMHC We proceed with bounding the components of the operator $\mathcal{M}_\epsilon$ from \eqref{defM}.\EMHC
\begin{Lem}[\BMHC properties of $\mathcal{M}_\epsilon$\EMHC ] \label{Lemma61}
For any $0<\epsilon\leq 1$ we have 
  \begin{align*}
||(\mathcal{M}_{\epsilon}G)_1||_2\leq C ||G||_2,\qquad
||(\mathcal{M}_{\epsilon}G)_2||_2\leq \epsilon^2C||G||_2
  \end{align*}
for \BMHC some constant $C$ independent of $\epsilon$ and all \EMHC $G\in (\mathsf{L}^2(\mathbb{R}))^2$.
\end{Lem}
\begin{proof}
Lemma \ref{Prop A} yields
\begin{align*}
||\mathcal{A}_{k_m\epsilon}(\mathcal{A}_{k_m\epsilon}W_*\mathcal{A}_{k_m\epsilon}G)||_2&\leq ||\mathcal{A}_{k_m\epsilon}W_*\mathcal{A}_{k_m\epsilon}G||_2\leq ||\mathcal{A}_{k_m\epsilon}W_*||_{\infty}||\mathcal{A}_{k_m\epsilon}G||_2\leq ||W_*||_{\infty}||G||_2
\end{align*}
for all $G\in \mathsf{L}^2(\mathbb{R})$ and the claim follows from the definition of $\mathcal{M}_{\epsilon}$ in \eqref{defM}.
\end{proof}
\BMHC We next observe that $\mathcal{L}_{\epsilon}$ from \eqref{vorfix} is a bounded linear operator on $(\mathsf{L}_{\rm{even}}^2(\mathbb{R}))^2$ and admits the adjoint \EMHC 
\begin{align*}
\mathcal{L}_{\epsilon}^* V=\begin{pmatrix}
\mathcal{B}_{\epsilon,11}V_1+\mathcal{B}_{\epsilon,21}V_2-\sum_{m}^{M}\mathcal{A}_{k_m\epsilon}(\mathcal{A}_{k_m\epsilon}W_*)(\eta_{1,1}^m\mathcal{A}_{k_m\epsilon}V_1+\epsilon^2\eta_{2,1}^m\mathcal{A}_{k\epsilon}V_2) \\
\mathcal{B}_{\epsilon,12}V_1+\mathcal{B}_{\epsilon,22}V_2-\sum_{m}^M\mathcal{A}_{k_m\epsilon}(\mathcal{A}_{k_m\epsilon}W_*)(\eta_{1,2}^m\mathcal{A}_{k_m\epsilon}V_1+\epsilon^2\eta_{2,2}^m\mathcal{A}_{k\epsilon}V_2) 
\end{pmatrix}.
  \end{align*} 
Passing \BMHC formally \EMHC to the limit $\epsilon\to0$ in \eqref{vorfix} we find \BMHC the limit operator $\mathcal{L}_0$ with\EMHC 
\begin{align*}
(\mathcal{L}_0\phi)_1&=(\phi_1+\lambda\phi_2)-(a_1+\lambda a_2)\phi''_1-(b_2+\lambda b_1) \phi''_2-2W_*(a_3+\lambda^2 a_4+\lambda a_5)\phi_1\\
&-2W_*(\lambda b_5+\lambda^2b_3+b_4)\phi_2,\\
(\mathcal{L}_0\phi)_2&=-c_2\phi_1+(\sigma_0-c_3)\phi_2
\end{align*}   
\BMHC thanks to \eqref{deflam} and \eqref{defM}, and the corresponding adjoint $\mathcal{L}_0^*$ reads \EMHC
    \begin{align*}
(\mathcal{L}_0^*\phi)_1&=\phi_1-(a_1+\lambda a_2)\phi''_1-2W_*(a_3+\lambda^2 a_4+\lambda a_5)\phi_1-c_2\phi_2,\\
(\mathcal{L}_0^*\phi)_2&=\lambda\phi_1-2W_*(\lambda b_5+\lambda^2b_3+b_4)\phi_1-(b_2+\lambda b_1) \phi''_1+(\sigma_0-c_3)\phi_2.
  \end{align*} 
\BMHC Notice that both $\mathcal{L}_0$ and $\mathcal{L}_0^*$ are unbounded operators
and that the above formulas are only well-defined if $\phi$ is sufficiently smooth. Moreover, for the computation of $\mathcal{L}_0^*$ we supposed that $\phi$ decays sufficiently fast.\EMHC  
\par
The operators $\mathcal{L}_{\epsilon}^*$ and $\mathcal{L}_0^*$ \BMHC play a prominent \EMHC role in the proof of the invertibility of $\mathcal{L}_{\epsilon}$. The next lemma shows that $\mathcal{L}_{\epsilon}^*\phi$ converges to $\mathcal{L}_0^*\phi$ \BMHC provided that $\phi$ is sufficiently nice.\EMHC 
\begin{Lem}[\BMHC pointwise convergence of $\mathcal{L}_\epsilon$\EMHC ]\label{Lemma 11}
For any function \BMHC $\phi\in (\mathsf{W}^{4,2}(\mathbb{R}))^2$ \EMHC we have
  \begin{align*}
\mathcal{L}^*_{\epsilon}\phi\quad \xrightarrow{\;\;\epsilon\rightarrow 0\;\;} \quad \mathcal{L}^*_0\phi
  \end{align*}
strongly in $(\mathsf{L}^2(\mathbb{R}))^2$.
\end{Lem}
\begin{proof}
The operator $\mathcal{L}^*_{\epsilon}$ consists of linear combinations of $\epsilon^{-2}(\id-\mathcal{A}_{k\epsilon}^2)$ and $\mathcal{A}_{k_m\epsilon}(\mathcal{A}_{k_m\epsilon}W_*)\mathcal{A}_{k_m\epsilon}$. Hence it suffices to show the \BMHC pointwise \EMHC $\mathsf{L}^2$-convergence of these operators for any test function $\phi\in \mathsf{W}^{4,2}(\mathbb{R})$. Using Lemma \ref{Prop A} we have 
 \begin{align*}
||\epsilon^{-2}(\id-\mathcal{A}_{k\epsilon}^2)\phi-(-\tfrac{1}{12}k^2 \phi'')||_2\leq C\epsilon^2||\phi''''||_2
\end{align*}
and 
\begin{align*}
||\mathcal{A}_{k\epsilon}\big((\mathcal{A}_{k\epsilon}W_*)(\mathcal{A}_{k\epsilon}\phi)\big)-&W_*\phi||_2\leq ||\mathcal{A}_{k\epsilon}\big((\mathcal{A}_{k\epsilon}W_*)(\mathcal{A}_{k\epsilon}\phi)\big)-\mathcal{A}_{k\epsilon}\big((\mathcal{A}_{k\epsilon}W_*)\phi \big)||_2+\\
&\qquad+||\mathcal{A}_{k\epsilon}\Big((\mathcal{A}_{k\epsilon}W_*)\phi \big)-\mathcal{A}_{k\epsilon}(W_*\phi)||_2+||\mathcal{A}_{k\epsilon}(W_*\phi)-W_*\phi||_2\\
&\leq ||(\mathcal{A}_{k\epsilon}W_*)(\mathcal{A}_{k\epsilon}\phi-\phi)||_2+||(\mathcal{A}_{k\epsilon}W_*-W_*)\phi||_2+||\mathcal{A}_{k\epsilon}(W_*\phi)-W_*\phi||_2\\
&\leq ||W_*||_{\infty}||\mathcal{A}_{k\epsilon}\phi-\phi||_2+||\phi||_{\infty}||\mathcal{A}_{k\epsilon}W_*-W_*||_2+\\
&+||\mathcal{A}_{k\epsilon}(W_*\phi)-W_*\phi||_2\\
&\leq Ck^2\epsilon^2( ||W_*||_{\infty}||\phi''||_2+||\phi||_{\infty}||W''_*||_2+||(W_*\phi)''||_2).
\end{align*}
The assertion follows immediately.
\end{proof}
\BMHC Another key ingredient to our asymptotic analysis is that the kernel of the limit operator $\mathcal{L}_0$ is simple. We also mention that a similar argument 
appears in all rigorous results on the KdV limit of FPU-type lattices and that the entire spectrum of $\mathcal{L}_0$ can be computed explicitly, see for instance \cite{FP99, FP04b} and the references therein.\EMHC
\begin{Lem}[\BMHC nullspace of $\mathcal{L}_0$\EMHC ]\label{Lemma 9}
Under Assumption \ref{Ass2}, the kernel of $\mathcal{L}_0$ in $(\mathsf{L}^2(\mathbb{R}))^2$ is spanned by $W'_0=(W'_*,\lambda W'_*)^T\in (\mathsf{L}_{odd}^2(\mathbb{R}))^2$, \BMHC where $W_*$ represents the KdV wave from \eqref{Ws} and \eqref{formelW}. \EMHC
\end{Lem}
\begin{proof}
\BMHC Let $(\phi_1,\phi_2)$ be in the kernel of $\mathcal{L}_0$. The second component of $\mathcal{L}_0\phi=0$ gives $\phi_2=\lambda\phi_1$, and substituting this back into the first one, we obtain \EMHC
\begin{align}\label{linear}
\phi''_1=d_1\phi_1-2d_2W_*\phi_1,
  \end{align}   
where $d_1$, $d_2$ are given by \eqref{defd}. Equation \eqref{linear} can be viewed as the linearization of \eqref{BW} and is solved by $\phi_1=W'_*\in \mathsf{L}_{odd}^2(\mathbb{R})$. 
\BMHC The uniqueness of $\phi_1$ (and hence of $\phi_2$) up to normalization follows from a classical Wronski argument for linear and planar ODEs, see \cite[Lemma 3.1]{HML15} for more details.\EMHC
\end{proof}
The next lemma shows \BMHC that \EMHC $R_{\epsilon}[W_{0}]$ and $\mathcal{P}_{\epsilon}[W_{0}]$ are uniformly bounded for small $\epsilon$. 
\begin{Lem}[\BMHC residuals of the KdV wave\EMHC]\label{R+P}
There exists a constant $C$ independent of $\epsilon$ such that 
\begin{align*}
||R_{\epsilon}[W_{0}]\BMHC ||_2+||\EMHC \mathcal{P}_{\epsilon}[W_{0}]||_2\leq C,
\end{align*}
for all $0<\epsilon\leq 1$.
\end{Lem}
\begin{proof}
\BMHC
Because $W_*$ and all its derivatives belong to $\mathsf{L}^2(\mathbb{R})\cap \mathsf{L}^{\infty}(\mathbb{R})$ according to \eqref{formelW}, the assertions 
can be verified by means of Lemma \ref{Prop A} and arguing as in the proof of Lemma \ref{Lemma 11}.\EMHC
\end{proof}
\subsection{Inverting the operator \texorpdfstring{$\mathcal{L}_{\epsilon}$}{L}}\label{sect:invL}
\BMHC Since $\mathcal{L}_0$ has a nontrivial kernel by Lemma \ref{Lemma 9}, we cannot expect to bound $\mathcal{L}_{\epsilon}^{-1}$ uniformly on $(\mathsf{L}^2(\mathbb{R}))^2$. We can, however, do so in the space of even functions. \EMHC
 \begin{Lem}[\BMHC parity properties\EMHC]
The subspace $(\mathsf{L}_{\rm{even}}^2(\mathbb{R}))^2$ is invariant in $(\mathsf{L}^2(\mathbb{R}))^2$ under the linear mapping $\mathcal{L}_{\epsilon}$.
\end{Lem}
\begin{proof}
According to the definitions \BMHC in \eqref{Eqn:DefBEps}, \eqref{vorfix}, and \eqref{defM}, \EMHC all components of $\mathcal{L}_{\epsilon}$ are linearly spanned by $\mathcal{A}_{k_m\epsilon}^2$ and $\big(\mathcal{A}_{k_m\epsilon}(\mathcal{A}_{k_m\epsilon}W_*)\big)\mathcal{A}_{k_m\epsilon}$. We know $\mathcal{A}_{k_m\epsilon}$ respects the even-odd parity and $W_*$ is an even function. The assertion follows.
\end{proof}
Now we are in a position to state and prove the main \BC asymptotic result. \EC
\begin{Thm}[invertibility of $\mathcal{L}_{\epsilon}$]\label{inverL}
There exists a constant $\epsilon_*>0$ such that for any $\epsilon\in (0,\epsilon_*)$ the operator $\mathcal{L}_{\epsilon}$ is invertible on $\mathsf{L}_{\rm{even}}^2(\mathbb{R})$. More precisely, there exists a constant \BMHC $C$ \EMHC which does not depend on $\epsilon$ such that
  \begin{align*}
||\mathcal{L}_{\epsilon}^{-1}G||_2\leq \BMHC C \EMHC ||G||_2
  \end{align*}
 for all $\epsilon\in (0,\epsilon_*)$ and all $G\in (\mathsf{L}_{\rm{even}}^2(\mathbb{R}))^2$.
\end{Thm}
\begin{proof}
\underline{Preliminaries}: We will show that there is a constant $c>0$ such that
  \begin{align}
\label{ineq}||\mathcal{L}_{\epsilon}V||_2\geq c||V||_2
  \end{align}
for all $V\in (\mathsf{L}_{\rm{even}}^2(\mathbb{R}))^2$ and all sufficiently small $\epsilon$. This implies that $\mathcal{L}_{\epsilon}$ is injective and has closed image. The same holds for the operator $\mathcal{T}_{\epsilon}^{-1}\circ \mathcal{L}_{\epsilon}$, \BMHC which is symmetric by Assumption \ref{Ass sym}, and we conclude that both $\mathcal{T}_{\epsilon}^{-1}\circ \mathcal{L}_{\epsilon}$ and $\mathcal{L}_{\epsilon}$ are surjective. Moreover,
\eqref{ineq} implies that the inverse of  $\mathcal{L}_{\epsilon}$ is a continuous operator.\EMHC
\par
\underline{Antithesis}: We suppose that \BMHC a constant $c$ as in \eqref{ineq} \EMHC does not exist. Then \BMHC we find \EMHC a sequence $(\epsilon_n)_{n\in \mathbb{N}}$ with $\lim_{n\rightarrow \infty} \epsilon_n=0$ as well as sequences $(V_n)_{n\in \mathbb{N}}$ and $(G_n)_{n\in \mathbb{N}}\subset (\mathsf{L}_{\rm{even}}^2(\mathbb{R}))^2$ such that
  \begin{align}
\label{inverL.Peqn8}
\mathcal{L}_{\epsilon_n}V_n=G_n,\qquad ||V_n||_2=1,\qquad ||G_n||_2\;\;\xrightarrow{\;\;n\rightarrow \infty\;\;}\;\; 0.
  \end{align}
In view of the invertibility of $\mathcal{B}_{\epsilon}$ we have 
 \begin{align}
\label{inverL.Peqn4}
V_n=\mathcal{B}_{\epsilon_n}^{-1}(\mathcal{M}_{\epsilon_n}V_n+G_n)=\begin{pmatrix}
(\mathcal{B}_{\epsilon_n}^{-1})_{11}(\mathcal{M}_{\epsilon_n}V_n+G_n)_1+(\mathcal{B}_{\epsilon_n}^{-1})_{12}(\mathcal{M}_{\epsilon_n}V_n+G_n)_2\\
(\mathcal{B}_{\epsilon_n}^{-1})_{21}(\mathcal{M}_{\epsilon_n}V_n+G_n)_1+(\mathcal{B}_{\epsilon_n}^{-1})_{22}(\mathcal{M}_{\epsilon_n}V_n+G_n)_2
\end{pmatrix},
\end{align}
where $(\mathcal{M}_{\epsilon_n}V_n+G_n)_1$, $(\mathcal{M}_{\epsilon_n}V_n+G_n)_2$ 
\BMHC denote the \EMHC components of $\mathcal{M}_{\epsilon_n}V_n+G_n$. Lemma \ref{Lemma61} gives
\begin{align}
\label{inverL.Peqn2a}
||(\mathcal{M}_{\epsilon_n}V_n+G_n)_1||_2&\leq C||V_n||_2+C||G_n||_2\leq C
\end{align}
and 
\begin{align}
\label{inverL.Peqn2b}
||(\mathcal{M}_{\epsilon_n}V_n+G_n)_2||_2&\leq C\epsilon^2||V_n||_2+||G_n||_2=C\epsilon_n^2+C||G_n||_2,
\end{align}
\BMHC 
where the right hand side in \eqref{inverL.Peqn2b} vanishes as $n\to\infty$.\EMHC 
\par
 \underline{Weak convergence to 0}: Since the sequence $(V_n)_{n\in \mathbb{N}}$ is bounded, there exists a $V_{\infty}\in (\mathsf{L}_{\rm{even}}^2(\mathbb{R}))^2$ such that 
\begin{align*}
V_n\;\;\xrightharpoonup{\;\;n\rightarrow \infty\;\;} \;\; V_{\infty} \quad \text{weakly in}  \quad (\mathsf{L}_{\rm{even}}^2(\mathbb{R}))^2
\end{align*}
\BMHC for a not relabelled subsequence. \EMHC
 Due to Lemma \ref{Lemma 11} we have
\begin{align}
\label{inverL.Peqn3}
\langle V_{\infty},\mathcal{L}_0^* \phi \rangle=\lim_{n\rightarrow \infty} \langle V_{n},\mathcal{L}_{\epsilon_n}^* \phi \rangle=\lim_{n\rightarrow \infty}\langle \mathcal{L}_{\epsilon_n}V_{n}, \phi \rangle=\lim_{n\rightarrow \infty}\langle G_n,\phi \rangle=0,
\end{align}
for \BMHC any test function $\phi$ that is sufficiently smooth  and localized. \EMHC  Let  $\phi_1$, $V_{\infty,1}$ denote the first components of $\phi$, $V_{\infty}$ and $\phi_2$, $V_{\infty,2}$ the \BMHC second ones. \EMHC Putting $\phi_1=0$ we obtain 
\begin{align}
\label{inverL.Peqn1}
\langle-c_2V_{\infty,1}+(\sigma_0-c_3)V_{\infty,2}, \phi_2\rangle=0,
\end{align}
for any $\phi_2$, and thanks to \BMHC \eqref{deflam} we get \EMHC
\begin{align*}
V_{\infty,2}=\frac{c_2}{\sigma_0-c_3}V_{\infty,1}=\lambda V_{\infty,1}.
\end{align*}
\EMHC Substituting this back into the equation $\langle V_{\infty},\mathcal{L}_0^* \phi\rangle=0$ and letting $\phi_2=0$ we obtain
\begin{align*}
\Big\langle V_{\infty,1},\  &(1+\lambda^2)\phi_1-\Big((a_1+\lambda a_2)+\lambda(b_2+\lambda b_1)\Big) \phi''_1\\
&-2W_*\Big((a_3+\lambda^2 a_4+\lambda a_5)+\lambda(\lambda b_5+\lambda^2b_3+b_4)\Big)\phi_1\Big\rangle=0,
  \end{align*}   
for any $\phi_1$, \BMHC and this gives \EMHC
\begin{align*}
\big|\langle V_{\infty,1},\phi''_1\rangle\big|\leq C||\phi_1||_2.
  \end{align*}   
\BMHC We conclude that \EMHC $V_{\infty,1}$ belongs to $\mathsf{W}^{2, 2}(\mathbb{R})$, and \eqref{inverL.Peqn1} ensures $V_{\infty}\in(\mathsf{W}^{2, 2}(\mathbb{R}))^2$. 
\BMHC We can hence apply $\mathcal{L}_0$ to $V_\infty$  and conclude in view of \eqref{inverL.Peqn3} that \EMHC
\begin{align*}
\langle \mathcal{L}_0 V_{\infty}, \phi\rangle=\langle V_{\infty},\mathcal{L}_0^* \phi\rangle=0
\end{align*}
\BMHC holds for any sufficiently nice test function $\phi$. In particular, the even function \EMHC  $V_{\infty}$ is contained in the kernel of $\mathcal{L}_0$ and \EMHC
\begin{align*}
V_{\infty}=0
\end{align*}
\BMHC is a direct consequence of Lemma \ref{Lemma 9}.\EMHC 
\par
 \underline{Further notations}: \BMHC In what follows we consider a constant $0<K<\infty$, denote by $\chi_K$ the characteristic function of the interval $I_K:=[-K,+K]$, and set
 \begin{align*}
d_K:=\sup_{|\xi|\geq K-k_{\max}}W_*(\xi),
\end{align*}
where $k_{\max} = \max_{m=1..M} k_m$. Moreover, in \EMHC order to estimate $V_n$ we break it up \BMHC via \EMHC
 \begin{align}
\label{inverL.Peqn5}
V_n=V_n^{(1)}+V_n^{(2)}+V_n^{(3)}
\end{align}
\BMHC with \EMHC
 \begin{align*}
V_n^{(1)}:=
\begin{pmatrix}
\chi_K\Pi_{\epsilon_n}V_{n,1}\\
\chi_K\Pi_{\epsilon_n}V_{n,2}
\end{pmatrix},
\qquad V_n^{(2)}:=
\begin{pmatrix}
(\id-\chi_K)\Pi_{\epsilon_n}V_{n,1}\\
(\id-\chi_K)\Pi_{\epsilon_n}V_{n,2}
\end{pmatrix},
\qquad V_n^{(3)}:=
\begin{pmatrix}
(\id-\Pi_{\epsilon_n})V_{n,1}\\
(\id-\Pi_{\epsilon_n})V_{n,2}
\end{pmatrix},
\end{align*}
\BMHC where $V_{n,1}$ and $V_{n,2}$ denote the components of $V_n$. Notice that
\begin{align*}
||V_n^{(j)}||_2\leq || V_n||_2\leq 1
\end{align*}
holds by construction for all $n$ and $j=1,2,3$.
\EMHC
\par
\BMHC  \underline{Choice of $K$}:  Since $V_n^{(2)}$ is supported in  $\mathbb{R}\setminus I_K$, the properties of the convolution operators $\mathcal{A}_{k_m\epsilon_n}$ imply 
\begin{align*}
||\mathcal{A}_{k_m\epsilon_n}\big(\mathcal{A}_{k_m\epsilon_n}W_*\mathcal{A}_{k_m\epsilon_n}
V_{n,i}^{(2)}\big)||_2&\leq ||\mathcal{A}_{k_m\epsilon_n}W_*\mathcal{A}_{k_m\epsilon_n}V_{n,i}^{(2)}||_2
\\&\leq
\Big(\sup_{|\xi|\geq K- k_{\max}/2} \big|(\mathcal{A}_{k_m\epsilon_n}W_*)(\xi)\big|\Big)||\mathcal{A}_{k_m\epsilon_n}V_{n,i}^{(2)}||_2\\&
\leq d_K||V_n^{(2)}||_2\leq d_K
\end{align*}
for $i=1,2$, $m=1..M$ and all $0<\epsilon_n\leq1$, and exploiting \eqref{defM} and Lemma \ref{Lemma 6} we obtain
\begin{align*}
||\mathcal{B}_{\epsilon_n}^{-1}\mathcal M_{\epsilon_n} V_n^{(2)}||_2 \leq C d_K
\end{align*}
for some constant $C$ independent of $K$. Consequently, choosing $K$ sufficiently large we can guarantee that
\begin{align}
\label{inverL.Peqn7}
||\mathcal{B}_{\epsilon_n}^{-1}\mathcal M_{\epsilon_n} V_n^{(2)}||_2 \leq \tfrac12
\end{align}
holds for all $n$.\EMHC
\par
\underline{Strong convergence of $V_n^{(1)}$ and $V_n^{(3)}$}:  \BMHC Lemma \ref{Lemma 6} combined with \eqref{inverL.Peqn4} yields
 \begin{align*}
||V_n^{(3)}||_2&\leq C\epsilon_n^2 ||(\mathcal{M}_{\epsilon_n}V_n+G_n)_1||_2+
||(\mathcal{M}_{\epsilon_n}V_n+G_n)_2||_2,
\end{align*}
so \eqref{inverL.Peqn2a} and \eqref{inverL.Peqn2b} provide
\begin{align}
\label{inverL.Peqn6a}
V_n^{(3)}\;\;\xrightarrow{\;\;n\rightarrow \infty\;\;}\;\; 0 \quad \text{strongly in} \quad (\mathsf{L}^2(\mathbb{R}))^2.
\end{align}
 To decompose  $V_n^{(1)}$ we define
 \begin{align*}
U_n^{(1)}:=
\begin{pmatrix}
\chi_K\Pi_{\epsilon_n}(\mathcal{B}_{\epsilon_n}^{-1})_{11}(\mathcal{M}_{\epsilon_n}V_n+G_n)_1\\
\chi_K\Pi_{\epsilon_n}(\mathcal{B}_{\epsilon_n}^{-1})_{21}(\mathcal{M}_{\epsilon_n}V_n+G_n)_1
\end{pmatrix},\qquad
\ U_n^{(2)}:=\begin{pmatrix}
\chi_K\Pi_{\epsilon_n}(\mathcal{B}_{\epsilon_n}^{-1})_{12}(\mathcal{M}_{\epsilon_n}V_n+G_n)_2\\
\chi_K\Pi_{\epsilon_n}(\mathcal{B}_{\epsilon_n}^{-1})_{22}(\mathcal{M}_{\epsilon_n}V_n+G_n)_2
\end{pmatrix}
\end{align*}
such that
 \begin{align*}
V_n^{(1)}=U_n^{(1)}+U_n^{(2)}.
\end{align*}
\BMHC By contruction, both $U_n^{(1)}$ and $U_n^{(2)}$ are supported in $I_K$, and from Lemma \ref{Lemma 6} we deduce the estimates
 \begin{align*}
||U_n^{(1)}||_{2,2,I_K}&\leq ||\Pi_{\epsilon_n}(\mathcal{B}_{\epsilon_n}^{-1})_{11}(\mathcal{M}_{\epsilon_n}V_n+G_n)_1||_{2,2}+||\Pi_{\epsilon_n}(\mathcal{B}_{\epsilon_n}^{-1})_{21}(\mathcal{M}_{\epsilon_n}V_n+G_n)_1||_{2,2}\leq C
\end{align*}
and 
 \begin{align*}
||U_n^{(2)}||_{2,I_K}&\leq ||\Pi_{\epsilon_n}(\mathcal{B}_{\epsilon_n}^{-1})_{12}(\mathcal{M}_{\epsilon_n}V_n+G_n)_2||_2+||\Pi_{\epsilon_n}(\mathcal{B}_{\epsilon_n}^{-1})_{22}(\mathcal{M}_{\epsilon_n}V_n+G_n)_2||_2\leq C(\epsilon_n^2+||G_n||_2)
\end{align*}
thanks to \eqref{inverL.Peqn2a} and \eqref{inverL.Peqn2b}, where the latter inequality shows that
\begin{align*}
U_n^{(2)}\;\;\xrightarrow{\;\;n\rightarrow \infty\;\;}\;\; 0 \quad \text{strongly in} \quad (\mathsf{L}^2(I_K))^2.
\end{align*}
\BMHC Moreover, since $V_n^{(2)}$ is supported in $\mathbb{R}\setminus I_K$ we have 
\begin{align*}
V_n^{(2)}\;\;\xrightarrow{\;\;n\rightarrow \infty\;\;}\;\; 0 \quad \text{strongly in} \quad (\mathsf{L}^2(I_K))^2
\end{align*}
and our results derived so far guarantee that
 \begin{align*}
V_n\xrightharpoonup{n\rightarrow \infty} 0 \quad \text{weakly in} \quad (\mathsf{L}^2(I_K))^2,\qquad\qquad
V_n^{(3)}\;\;\xrightarrow{\;\;n\rightarrow \infty\;\;}\;\; \quad \text{strongly in} \quad (\mathsf{L}^2(I_K))^2.
\end{align*}
Combing this with 
\begin{align*}
V_n=U_n^{(1)}+U_n^{(2)}+V_n^{(2)}+V_n^{(3)}
\end{align*}
and using that $\mathsf{W}^{2,2}(I_K)$ is compactly embedded into $\mathsf{L}^2(I_K)$, we get
\begin{align*}
U_n^{(1)}\;\;\xrightarrow{\;\;n\rightarrow \infty\;\;}\;\; 0 \quad \text{strongly in} \quad (\mathsf{L}^2(I_K))^2
\end{align*}
and hence
\begin{align*}
V_n^{(1)}\;\;\xrightarrow{\;\;n\rightarrow \infty\;\;}\;\; 0 \quad \text{strongly in} \quad (\mathsf{L}^2(I_K))^2.
\end{align*}
This finally implies
\begin{align}
\label{inverL.Peqn6b}
V_n^{(1)}\;\;\xrightarrow{\;\;n\rightarrow \infty\;\;}\;\; 0 \quad \text{strongly in} \quad (\mathsf{L}^2(\mathbb{R}))^2
\end{align}
since $V_n^{(1)}$ is supported $I_K$.
\par
\underline{Derivation of the contradiction}: By \eqref{inverL.Peqn4} and \eqref{inverL.Peqn5} we have
\begin{align*}
||V_n||_2\leq ||\mathcal{B}_{\epsilon_n}^{-1}\mathcal{M}_{\epsilon_n}V_n^{(1)}||_2+
||\mathcal{B}_{\epsilon_n}^{-1}\mathcal{M}_{\epsilon_n}V_n^{(2)}||_2+
||\mathcal{B}_{\epsilon_n}^{-1}\mathcal{M}_{\epsilon_n}V_n^{(3)}||_2+
||\mathcal{B}_{\epsilon_n}^{-1}G_n||_2
\end{align*}
and passing to the limit $n\to \infty$ gives
 \begin{align*}
\limsup_{n\rightarrow \infty}||V_n||_2\leq \tfrac{1}{2}
\end{align*}
thanks to the Lemmas \ref{Lemma 6} and \ref{Lemma61},  and according to \eqref{inverL.Peqn7}, \eqref{inverL.Peqn6a}, and \eqref{inverL.Peqn6b}. This, however, contradicts the normalization condition in \eqref{inverL.Peqn8} and the antithesis must be false.\EMHC
\end{proof}
\subsection{Nonlinear fixed-point argument}
\label{sect:FP}
To conclude the existence proof for lattice waves it remains to verify  the conditions of the Banach fixed-point theorem.
\begin{Thm}[\BMHC existence and uniqueness of lattice waves\EMHC]
Under Assumptions \ref{Ass sym}, \ref{Ass2}, \ref{Ass inverse} and \ref{Ass1} there exist constants $D>0$ and $\epsilon_*>0$, such that for any $\epsilon<\epsilon_*$ the operator $\mathcal{F}_{\epsilon}$ has a unique fixed point in the ball \BC $B_D:=\{V\in (\mathsf{L}_{\rm{even}}^2(\mathbb{R}))^2: ||V||_2\leq D\}$.\EC
\end{Thm}
\begin{proof}
We \BMHC show \EMHC that the operator $\mathcal{F}_{\epsilon}$ is a contraction in a sufficiently large ball $B_D$ for any sufficiently small $\epsilon$, thus satisfying all the conditions of the Banach Fixed Point Theorem. \BMHC The value of $D$ will be chosen later.\EMHC
\par
\underline{Estimates for the quadratic terms:} For arbitrary $V_1,\ V_2 \in B_D$, Assumption \ref{Ass1} gives 
\begin{align*}
 ||\mathcal{A}_{k_m\epsilon}(\mathcal{A}_{k_m\epsilon}V_{1,i})^2-\mathcal{A}_{k_m\epsilon}(\mathcal{A}_{k_m\epsilon}V_{2,i})^2||_2&\leq ||(\mathcal{A}_{k_m\epsilon}V_{1,i})^2-(\mathcal{A}_{k_m\epsilon}V_{2,i})^2||_2\\
&\leq \big(||\mathcal{A}_{k_m\epsilon}V_{1,i}||_{\infty}+||\mathcal{A}_{k_m\epsilon}V_{2,i}||_{\infty}\big)||\mathcal{A}_{k_m\epsilon}(V_{1,i}-V_{2,i})||_2\\
&\leq k_m^{-1/2}\epsilon^{-1/2}\big(||V_{1,i}||_{2}+||V_{2,i}||_{2}\big)||V_{1,i}-V_{2,i}||_2\\
&\leq 2Dk_m^{-1/2}\epsilon^{-1/2}||V_{1}-V_{2}||_2
\end{align*}
\BMHC as well as \EMHC
\begin{align*}
 & ||\mathcal{A}_{k_m\epsilon}\big((\mathcal{A}_{k_m\epsilon}V_{1,1})(\mathcal{A}_{k_m\epsilon}V_{1,2})\big)-\mathcal{A}_{k_m\epsilon}\big((\mathcal{A}_{k_m\epsilon}V_{2,1})(\mathcal{A}_{k_m\epsilon}V_{2,2})\big)||_2
\\&\leq ||\mathcal{A}_{k_m\epsilon}V_{1,1}||_{\infty}||\mathcal{A}_{k_m\epsilon}V_{1,2}-\mathcal{A}_{k_m\epsilon}V_{2,2}||_2+||\mathcal{A}_{k_m\epsilon}V_{2,2}||_{\infty}||\mathcal{A}_{k_m\epsilon}V_{1,1}-\mathcal{A}_{k_m\epsilon}V_{2,1}||_2\\
&\leq k_m^{-1/2}\epsilon^{-1/2}\Big(||V_{1,1}||_{2}||V_{1,2}-V_{2,2}||_2+||V_{2,2}||_{2}||V_{1,1}-V_{2,1}||_2\Big)\\
&\leq Dk_m^{-1/2}\epsilon^{-1/2}||V_{1}-V_{2}||_2.
\end{align*}
Applying this to $\mathcal{Q}_{\epsilon}$ yields
\begin{align*}
||\epsilon^2\mathcal{Q}_{\epsilon}[V_1]-\epsilon^2\mathcal{Q}_{\epsilon}[V_2]||_2\leq CD\epsilon^{3/2}||V_{1}-V_{2}||_2
\end{align*}
for $V_1, V_2\in B_D$, and \BMHC setting $V_1=V$ and $V_2=0$ we get\EMHC
\begin{align*}
||\epsilon^2\mathcal{Q}_{\epsilon}[V]||_2\leq CD\epsilon^{3/2}||V||_2\leq CD^2\epsilon^{3/2}
\end{align*}
for any $V\in B_D$. We emphasize that $C$ is a constant, which does not depend on $D$.
\par
\underline{Estimates for the higher order terms:} \BMHC Assumption \ref{Ass1} combined with Lemma \ref{Prop A} provides
\begin{align*}
&\ \ \ \ \Big|\Big|\Psi_i^m\big(\epsilon^2k_m\mathcal{A}_{k_m\epsilon}(W_*+\epsilon^2V_{1})\big)-
\Psi_i^m\big(\epsilon^2k_m\mathcal{A}_{k_m\epsilon}(W_*+\epsilon^2V_{1})\big)\Big|\Big|_2\\
&\leq \gamma_i^mk_m^{-1}\epsilon^7\Big(||W_*+\epsilon^2V_{1}||_{2}^2+||W_*+\epsilon^2V_{2}||_{2}^2\Big)||V_{1}-V_{2}||_2\leq C\gamma_i^m\epsilon^7(C+\epsilon^4D^2)||V_1-V_2||_2
\end{align*}
\EMHC for $V_1, V_2\in B_D$. This implies 
\begin{align*}
||\epsilon^2\mathcal{N}_{\epsilon}[W_{0};V_1]-\epsilon^2\mathcal{N}_{\epsilon}[W_{0};V_2]||_2&=||\mathcal{P}_{\epsilon}[W_{0}+V_1]-\mathcal{P}_{\epsilon}[W_{0}+V_2]||_2\leq \epsilon C(C+\epsilon^4D^2)||V_1-V_2||_2
\end{align*}
 for $V_1, V_2\in B_D$, \BMHC and also \EMHC
 \begin{align*}
||\epsilon^2\mathcal{N}_{\epsilon}[W_{0};V]||_2\leq \epsilon C(C+\epsilon^4D^2)||V||_2\leq \epsilon CD(C+\epsilon^4D^2)
\end{align*}
for all $V\in B_D$ thanks to \BMHC $\mathcal{N}_{\epsilon}[W_{0};0]=0$. \EMHC
\par
\underline{Concluding arguments:} Lemma \ref{R+P} finally gives 
\begin{align*}
||R_{\epsilon}[W_{0}]+\mathcal{P}_{\epsilon}[W_{0}]||_2\leq C,
\end{align*}
\BMHC and this implies \EMHC 
\begin{align*}
||\mathcal{F}_{\epsilon}[V]||_2\leq C+CD^2\epsilon^{3/2}+\epsilon CD(C+\epsilon^4D^2)
\end{align*}
for any $V\in B_D$. Moreover, the above estimates also ensure that \EMHC
\begin{align*}
||\mathcal{F}_{\epsilon}[V_1]-\mathcal{F}_{\epsilon}[V_2]||_2\leq \big(CD\epsilon^{3/2}+\epsilon C(C+\epsilon^4D^2)\big)||V_1-V_2||_2
\end{align*}
for all $V_1,V_2\in B_D$. Now we choose $D$ to be $2C$ \BMHC and find afterwards a value \EMHC $\epsilon_*>0$ such that 
\begin{align*}
C+CD^2\epsilon_*^{3/2}+\epsilon_* CD(C+\epsilon_*^4D^2)\leq D\qquad \text{and}\qquad
 \Big(CD\epsilon_*^{3/2}+\epsilon_* C(C+\epsilon_*^4D^2)\Big)< 1.
\end{align*}
\BMHC The claim is now granted by Banach's Contraction Mapping Principle.\EMHC
\end{proof}
\section{Applications to different lattices}
\label{sect:app}
\BMHC In this section we discuss the implication of our abstract existence result from \S\ref{sect:abstract} for the three lattice geometries sketched in Figure \ref{Fig:AllLattices}. \EMHC  The existence of a solitary solution as given by Theorem \ref{maintm} involves the verification of Assumptions \ref{Ass sym}, \ref{Ass2}, \ref{Ass inverse} and \ref{Ass1}. Notice that since all the relevant effective potentials are smooth at the point $(0,0)$, Assumptions \ref{Ass sym} and \ref{Ass1} are automatically satisfied. Thus only two assumptions remain to be verified, \BMHC and this \EMHC will be done numerically. For an intuitive picture we \BMHC also \EMHC give the behavior \BMHC of \EMHC $W_0$ and $\lambda$ with respect to \BMHC the propagation direction, i.e., the angle $\alpha$. All simulations are performed with $r_*=0.8047$ and $V_1(r)=V_2(r)=V(r)=\tfrac{1}{2}(r-1)^2$. \EMHC
\begin{figure}[ht!]
\begin{center}
   \includegraphics[scale=0.25]{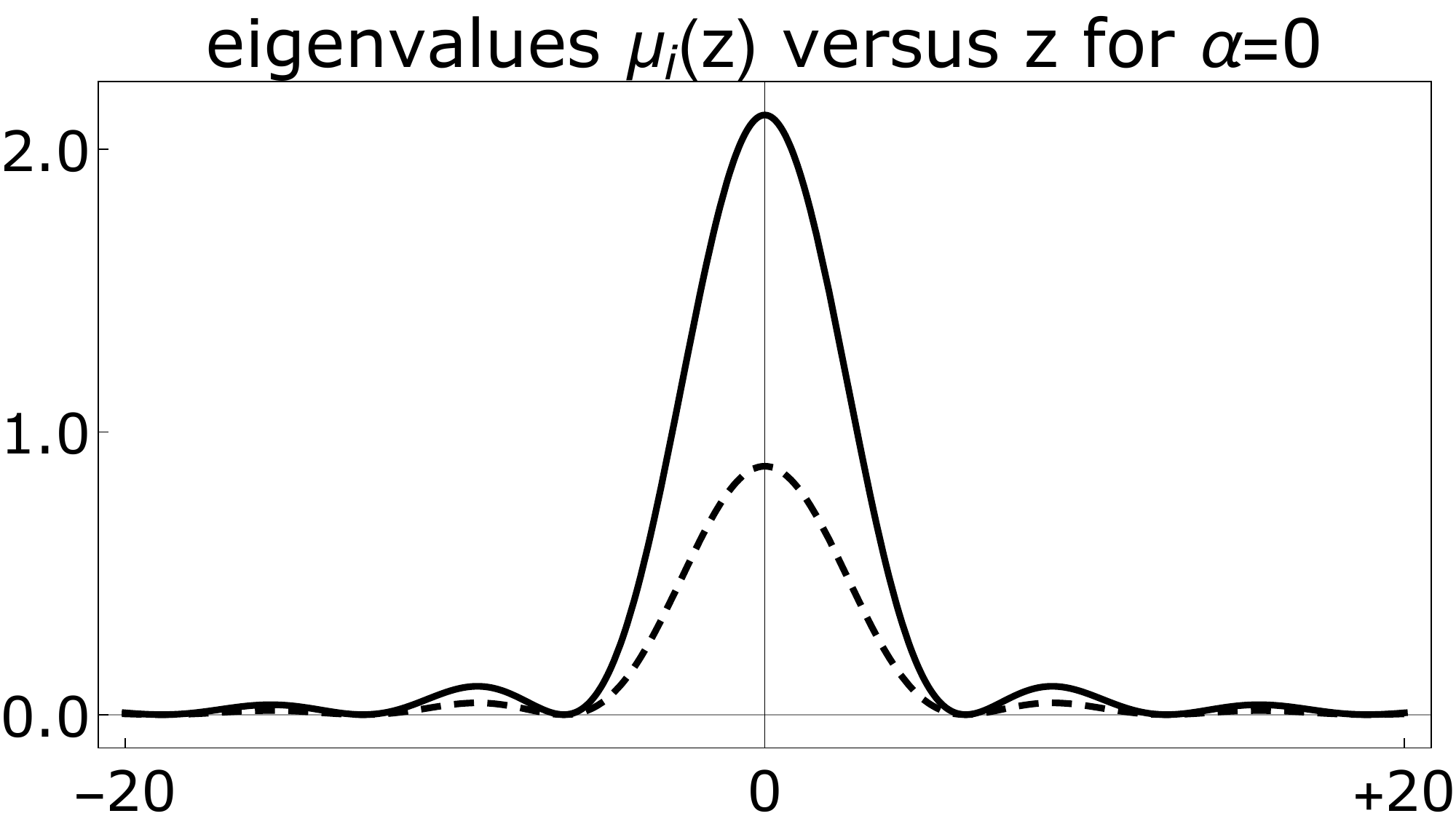}
\end{center}
\caption{\BMHC The auxiliary functions $\mu_1$ and $\mu_2$ from \eqref{Eqn:Mu} in solid and dashed lines, respectively, for the square lattice with $\alpha=0$. \EMHC}
\label{Fig:AuxFuntions}
\end{figure}
\par
We first remind that \eqref{eqsigma} gives two possible candidates for $\sigma_0$, but in our approach we always chose the larger solution, see \eqref{defsig}. The reason is that although the other solution branch might give rise to a well-defined KdV wave in the limit $\epsilon\to0$, we cannot expect the corresponding operator $\mathcal{B}_\epsilon$ to be invertible. To see this, we observe that $\mathcal{B}_\epsilon$ is invertible if and only if 
$\sigma_\epsilon$ does not belong to the spectrum of the operator 
$\mathcal{J}_{\epsilon}:=\sigma_{\epsilon}\id- \epsilon^2\mathcal{B}_{\epsilon}^{\text{can}}$, \BMHC see formulas \eqref{Bcan} and \eqref{Eqn:DefBEps}. Since the Fourier symbols of the components of $\mathcal{J}_{\epsilon}$ are given by \EMHC
\begin{align*}
\hat{\mathcal{J}}_{\epsilon,ij}(z)=\sum_{m=1}^Mk_m^2\alpha_{i,j}^m{\sinc}^2 (k_m\epsilon z/2)\,,
\end{align*}
we conclude that
\begin{align}
\label{Eqn:Mu}
\text{spec}\mathcal{J}_{\epsilon}=\bigcup_{i=1,2}\big\{\mu_i(\epsilon z): \epsilon z\in \mathbb{R}\big\}\,,
\end{align}
where the functions $\mu_1$, $\mu_2$ \BMHC can be computed pointwise in $\epsilon z$ \EMHC and are completely determined by the coefficients $k_m$ and $\alpha_{i,j}^m$, see Figure \ref{Fig:AuxFuntions} for an illustration.  Moreover, by \BMHC condition \eqref{eqsigma} \EMHC we have $\sigma_0\in\{\mu_1(0),\,\mu_2(0)\}$, \BMHC where our choice in 
\eqref{defsig} implies $\sigma_0=\max\{\mu_1(0),\,\mu_2(0)\}$ and \EMHC
Assumption \ref{Ass inverse} guarantees via
\begin{align*}
\max\limits_{i\in\{1,2\}\,,\; z\in\mathbb{R}} \mu_i(z)=
\sigma_0<\sigma_\epsilon
\end{align*} 
the desired invertibility of $\mathcal{B}_{\epsilon}$. For 
$\sigma_0=\min\{\mu_1(0),\,\mu_2(0)\}$, however, 
Assumption \ref{Ass inverse} can generically not be satisfied. In this case we expect that KdV-type waves still exist, but exhibit oscillatory tails due to resonances with the continuous spectrum.  A proof of this fact would involve sophisticated arguments lying beyond the scope of the present paper, but we refer to \cite{HMSZ13,VSWP16,HW17} for similar rigorous results on certain generalizations of one-dimensional FPU chains.
\subsection{Square lattice}
\BMHC We continue the discussion of the square lattice and refer to \S\ref{sect:Overview} for the expressions of the effective potentials and the parameters $k_m$.\EMHC
\begin{figure}[ht!]
\centering{%
\includegraphics[width=0.95\textwidth]{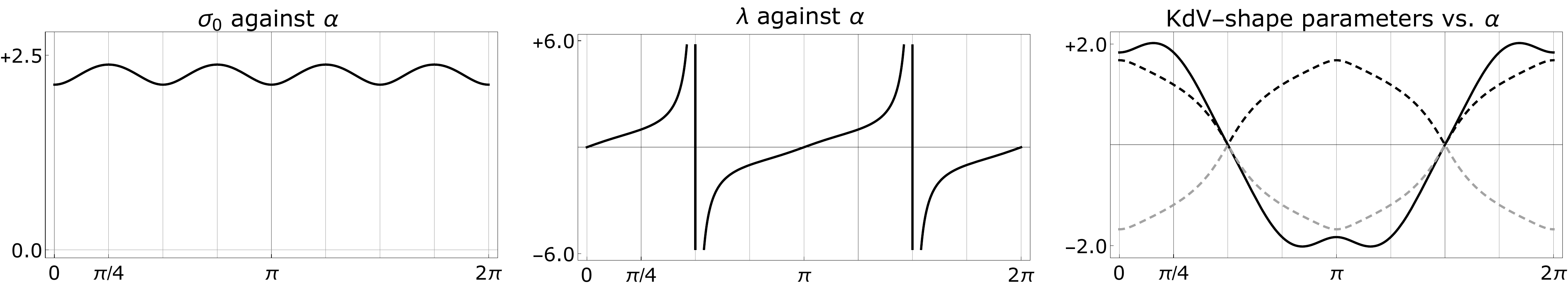}
}%
\caption{\BMHC KdV-wave parameters for the square lattice. \EMHC \textit{Left and center}. The behavior of $\sigma_0$, $\lambda$ with respect to $\alpha$. \textit{Right}. \BMHC Shape parameters $p_1$ and $p_2$, where the positive and the \EMHC negative roots of $p_2^2$ are given in dashed and gray lines respectively. }
\label{figsqu}
\end{figure}
\begin{figure}[ht!]
\centering{ 
 \includegraphics[width=0.95\textwidth]{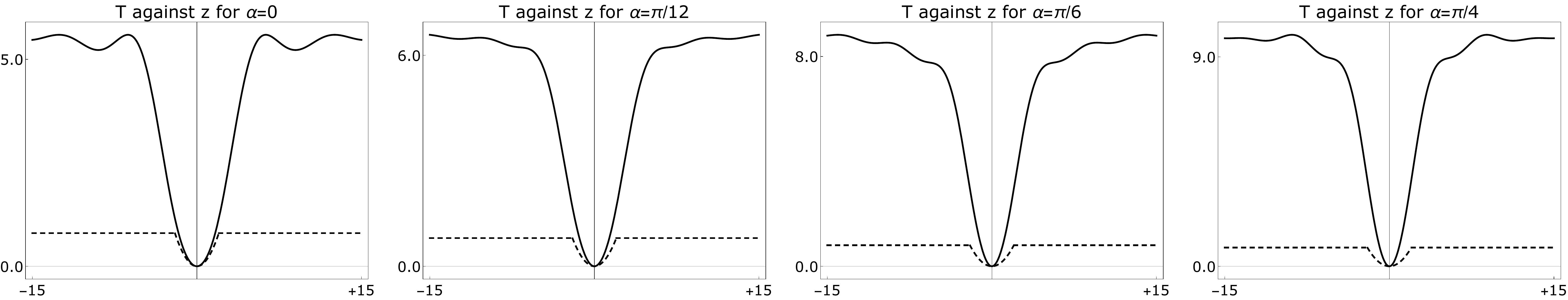}
}
\caption{\BMHC Parameter test for the square lattice. \EMHC $T(z)$ (solid) and $g(z)=0.3\cdot (\min\{z, 2\})^2$ (dashed)  for several values of $\alpha$. Assumption \ref{Ass inverse} requires $T(z)\geq g(z)$ for all $z\in \mathbb{R}$.}
\label{squT}
\end{figure}
\begin{figure}[ht!]
\centering{ 
  \includegraphics[width=0.95\textwidth]{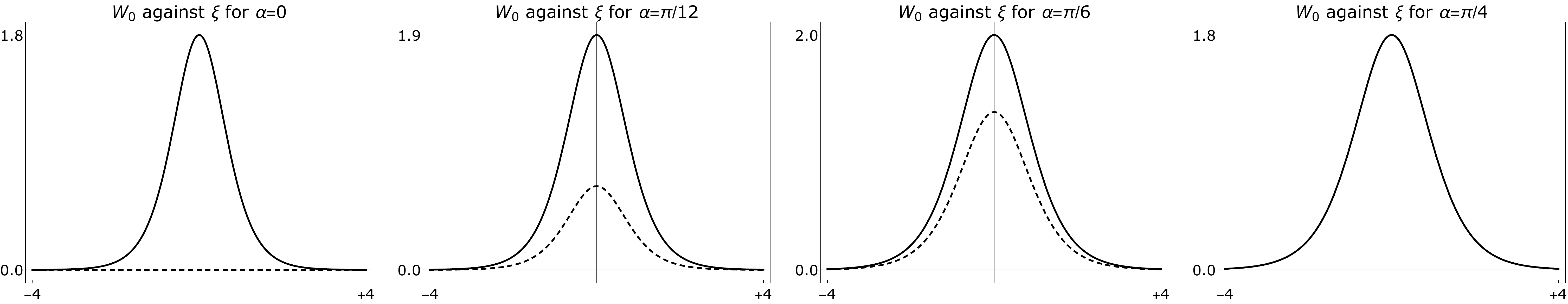}
}%
\caption{\BMHC KdV-limit profiles for selected values of $\alpha$ in the square lattice, where the first and the second component of $W_0$ are represented by the solid and the dashed lines, respectively. }
\label{squW}
\end{figure}
\par
\underline{Test of Assumption \ref{Ass2}:} 
The purpose of Assumption \ref{Ass2} is to ensure the existence of KdV traveling waves in the formal limit $\epsilon\rightarrow 0$. This is actually a minimal condition for the existence of a solitary wave for $\epsilon>0$. 
In Figure \ref{figsqu} we see that $\sigma_0$ is always positive \BMHC but non-constant and that
both $\lambda$ and the shape parameters $p_i$ from \eqref{Eqn:ShapePrms} are well-defined and non-vanishing for all values of $\alpha$ except for $\alpha=\pm\pi/2$, \BC and this implies that the parameters $d_i$ from \eqref{defd} satisfy the second part of
Assumption \ref{Ass2}. The singularities at $\alpha=\pm\pi/2$, \EC however,
have no geometric meaning but reflect that the corresponding wave profiles extend via $W_1=0\neq W_2$ only in the vertical direction, see \eqref{Ws1}. 
\par
\underline{Test of Assumption \ref{Ass inverse}:} We list in Figure \ref{squT} the numerically computed function $T$ for \BMHC $\alpha\in\{0,\,\tfrac{\pi}{12},\, \tfrac{\pi}{6},\tfrac{\pi}{4}\}$,
 and recall that \BMHC this function \EMHC is quadratic near $z=0$, see the remark to Assumption \ref{Ass inverse}. \EMHC These graphs \BMHC indicate that the required condition is satisfied for all values of $\alpha$.\EMHC
\par
\underline{\BMHC Velocity \EMHC profiles:} \BMHC To illustrate the direction dependence of the lattice waves, we display \EMHC in Figure \ref{squW} the two components of the \BMHC limiting velocity profiles $W_0=\lim_{\epsilon\to0} W_\epsilon$ \EMHC for the same set of values of $\alpha$ as in Figure \ref{squT}. \BMHC For \EMHC $\alpha=\frac{\pi}{4}$ we observe the coincidence of the two components of $W_0$, which is consistent with the axial symmetry of the lattice with respect to the diagonal direction. \BMHC We also recall that the case $\alpha=0$ has been studied in \cite{FM03}.\EMHC
\subsection{Diamond lattice}
\BMHC The diamond lattice from the right panel in Figure \ref{Fig:AllLattices} \EMHC arises by taking out the springs on one diagonal of the square lattice and then rotating by $45$ degrees. For each particle three pairs of forces come into question, \BMHC and thus we have $M=3$.
Moreover, similarly to the discussion in \S\ref{sect:Overview}, \EMHC the values for $k_m$ are obtained by \BMHC projecting the propagation vector \EMHC to the axes of symmetry. These are $(0,1)$, $(\frac{1}{\sqrt{2}}, \frac{1}{\sqrt{2}})$ and  $(\frac{1}{\sqrt{2}}, -\frac{1}{\sqrt{2}})$ and \BMHC provide \EMHC
\begin{align*}
k_1=\sin(\alpha),\quad k_2=\frac{1}{\sqrt{2}}\cos(\alpha)+\frac{1}{\sqrt{2}}\sin(\alpha),\quad k_3=\frac{1}{\sqrt{2}}\cos(\alpha)-\frac{1}{\sqrt{2}}\sin(\alpha)
\end{align*}
\BMHC as well as the effective potentials
\begin{align*}
\phi_1(x_1,x_2)& =V(\sqrt{x_1^2+(x_2+r_*)^2}), \quad \phi_2(x_1,x_2)=V(\sqrt{(x_1+\frac{1}{\sqrt{2}}r_*)^2+(x_2+\frac{1}{\sqrt{2}}r_*)^2}), \\
\phi_3(x_1,x_2)& =V(\sqrt{(x_1+\frac{1}{\sqrt{2}} r_*)^2+(x_2-\frac{1}{\sqrt{2}} r_*)^2}),
\end{align*}
where we supposed for simplicity that all springs can be described by the same potential $V$.
Finally, notice that \EMHC the number of symmetries is considerably reduced as compared with the square lattice: we still have the two reflection symmetries and but no more invariance under rotations by 45 degrees.
\par
\begin{figure}[!ht]
\centering{%
\includegraphics[width=0.95\textwidth]{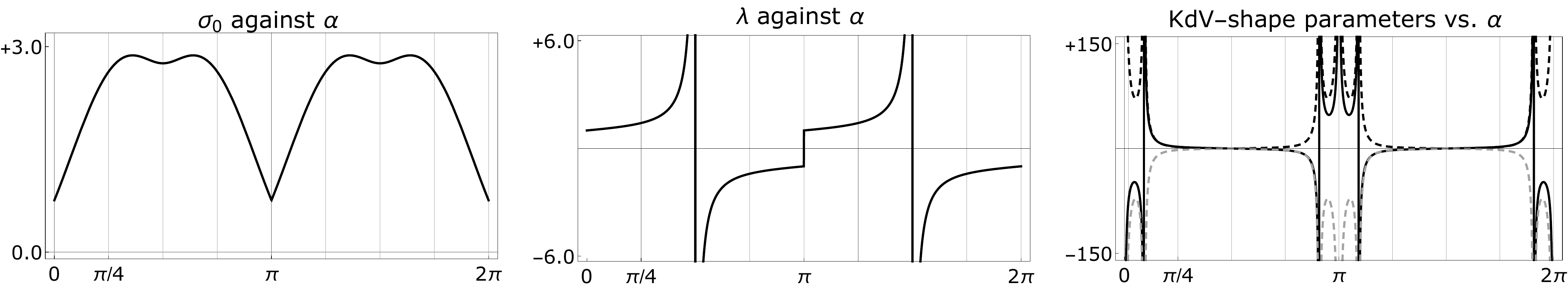}
}%
\caption{The plots from Figure \ref{figsqu} for the diamond lattice. \BMHC In \EMHC the graph of $\lambda$ we find jumps at multiples of $\pi$, which is consistent with the fact that the lattice is symmetric with respect to the horizontal direction. For $\alpha=0$ no KdV wave exists due to \BMHC this \EMHC singularity. }
\label{figdia}
\end{figure}
\begin{figure}[ht!]
\centering{ 
 \includegraphics[width=0.95\textwidth]{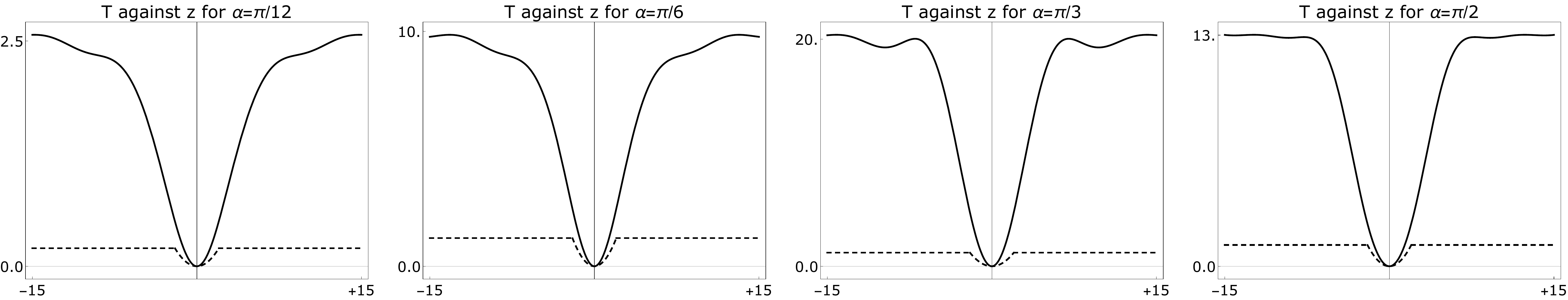}
}
\caption{The plots from Figure \ref{squT} for the diamond lattice. }
\label{diaT}%
\end{figure}
\begin{figure}[ht!]
\centering{ 
   \includegraphics[width=0.95\textwidth]{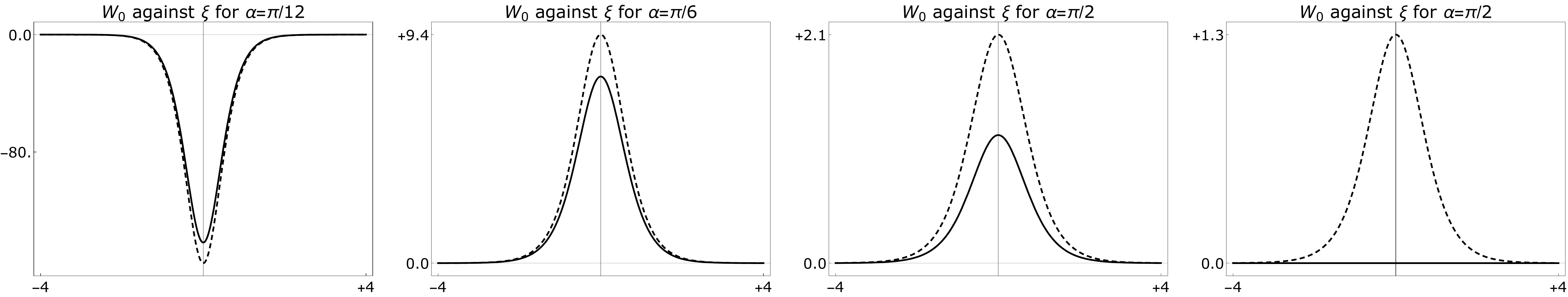}
   }
\caption{The plots from Figure \ref{squW} for the diamond lattice. }
\label{diaW} 
\end{figure}
\underline{Test of Assumption \ref{Ass2}:}  In Figure \ref{figdia} we see that \BMHC the graph of $\sigma_0$ is always positive but exhibits considerable changes  instead of the smooth and oscillatory behavior in the square lattice. \EMHC Due to the axial symmetry of the lattice
it suffices to consider $\alpha\in [0,\frac{\pi}{2}]$, and in the graph of KdV-shape parameters it is easily seen that $p_1\neq 0$  and \BC $p_2\neq0$ \EC can be defined for all $\alpha\in (0,\frac{\pi}{4}]$ except for \BMHC  isolated critical values.\EMHC 
\par
\underline{Test of Assumption \ref{Ass inverse}:} \BMHC Figure \ref{diaT} shows the numerical graphs of $T$ \EMHC for $\alpha\in\{\frac{\pi}{12},\ \frac{\pi}{6},\ \frac{\pi}{3},\ \frac{\pi}{2}\}$, \BMHC where we omitted the plot for \EMHC $\alpha=0$ since the KdV-shape parameters explode for this value. Similarly as before, Assumption \ref{Ass2} is fulfilled locally around $z=0$ and Figure \ref{diaT} gives numerical evidence for its global validity. 
\par
\underline{\BMHC Velocity \EMHC profiles:}  We observe in Figure \ref{diaW} a clear flip of the velocity profiles, when $\alpha$ turns from $\frac{\pi}{12}$ to $\frac{\pi}{6}$. This means that the wave is compressive for small values of $\alpha$ and above a certain value it becomes expansive. The threshold for this change is exactly the singularity between $0$ and $\frac{\pi}{2}$ as seen in the $(p_1,\ p_2)$--$\alpha$ graph of Figure \ref{figdia}. \BMHC For \EMHC $\alpha=\frac{\pi}{2}$ we observe the vanishing of the first component of $W_0$, which hints at a similar result as proven in \cite{FM03}, namely the existence of a unidirectional and longitudinal KdV-like solitary wave, whose first component completely vanishes. 
\subsection{Triangle lattice}
\BMHC In this example we have $M=3$ with
\begin{align*}
k_1=\cos(\alpha),\quad k_2=\frac{1}{2}\cos(\alpha)+\frac{\sqrt{3}}{2}\sin(\alpha),\quad k_3=\frac{1}{2}\cos(\alpha)-\frac{\sqrt{3}}{2}\sin(\alpha),
\end{align*}
so the effective potentials are given by
\begin{align*}
\phi_1(x_1,x_2)& =V(\sqrt{(x_1+r_*)^2+x_2^2}), \quad \phi_2(x_1,x_2) =V(\sqrt{(x_1+\frac{1}{2}r_*)^2+(x_2+\frac{\sqrt{3}}{2}r_*)^2}), \\
\phi_3(x_1,x_2)& =V(\sqrt{(x_1+\frac{1}{2}r_*)^2+(x_2-\frac{\sqrt{3}}{2}r_*)^2}).
\end{align*}
\EMHC
\begin{figure}[ht!]
\centering{%
\includegraphics[width=0.95\textwidth]{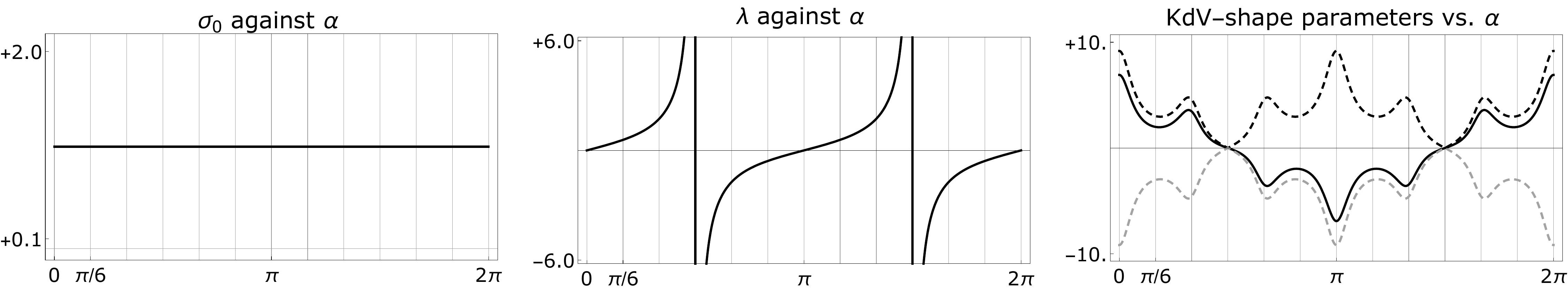}
}%
\caption{The plots from Figure \ref{figsqu} for the triangle lattice. }
\label{figtri}
\end{figure}
\begin{figure}[ht!]
\centering{ 
 \includegraphics[width=0.95\textwidth]{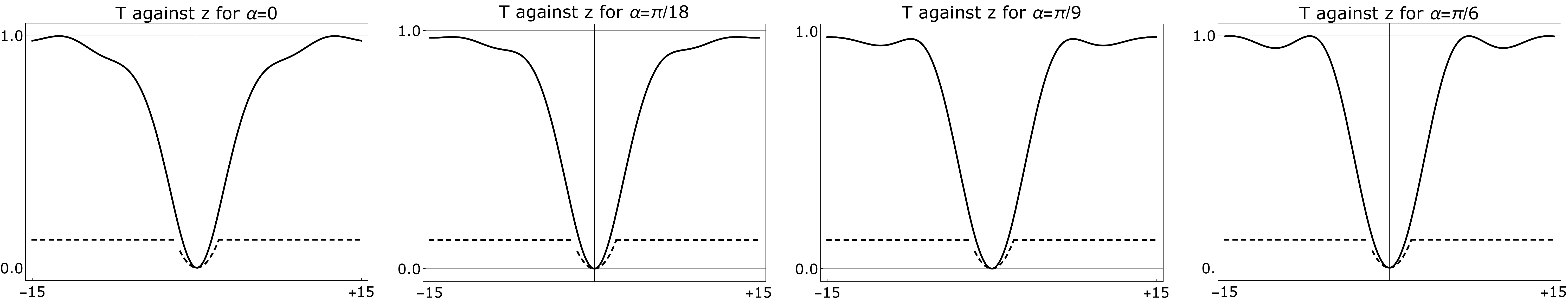}
}
\caption{The plots from Figure \ref{squT} for the triangle lattice. }
\label{triT}%
\end{figure}
\begin{figure}[ht!]
\centering{ 
   \includegraphics[width=0.95\textwidth]{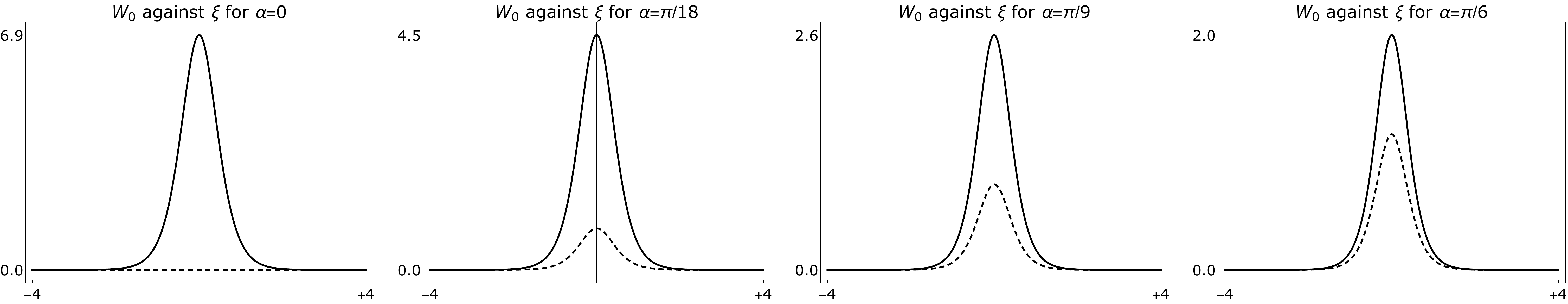}
   }
\caption{The plots from Figure \ref{squW} for the triangle lattice. }
\label{triW} 
\end{figure}
\underline{Test of Assumption \ref{Ass2} and \ref{Ass inverse}:}  As shown in Figure \ref{figtri}, the sound speed $\sqrt\sigma_0$ is constant for all $\alpha$. By comparing Figures \ref{figsqu}, \ref{figdia} \BMHC we might guess that the $\alpha$-dependence of $\sigma_0$  is more dramatic in lattices with few symmetries, but we have no rigorous explanation for this observation. Moreover, the plots in Figure \ref{triT} indicate that
Assumption \ref{Ass inverse} is satisfied for all values of $\alpha$.\EMHC
\par
\underline{\BMHC Velocity \EMHC profiles:} \BMHC Figure \ref{triW} illustrates the existence of
expansion waves in all directions and for $\alpha\in2\pi\{0,1/3,2/3\}$ we find again
unidirectional waves which propagate longitudinal with respect to one of the principal axes in the lattice.\EMHC
\section*{Acknowledgments}
 Both authors authors are grateful for the financial support by the \emph{Deutsche Forschungsgemeinschaft} (DFG individual grant HE 6853/2-1). They also sincerely thank the referee for careful reading and for the substantial list of very helpful comments concerning the exposition of the material and the presentation of the technical details.

%
%

%
%
%
\end{document}